\documentclass[12pt]{amsart}
\usepackage{times,amssymb}
\usepackage{comment}
\usepackage[usenames, dvipsnames]{color}
\usepackage{hyperref}

\parskip=\smallskipamount

\newtheorem*{Thm*}{Theorem}
\newtheorem{Thm}{Theorem}
\newtheorem{Cor}[Thm]{Corollary}
\newtheorem{Prop}[Thm]{Proposition}
\newtheorem{Lemma}[Thm]{Lemma}

\theoremstyle{definition}
\newtheorem{Defn}[Thm]{Definition}
\newtheorem{Notation}[Thm]{Notation}

\newtheorem{Remark}[Thm]{Remark}
\newtheorem{Example}[Thm]{Example}
\newtheorem{Construction}[Thm]{Construction}

\newcommand{\abs}[1]{\left\vert#1\right\vert}

\newcommand{\set}[1]{\left\{#1\right\}}
\newcommand{\mf}[1]{\mathbb{#1}}

\newcommand{\mc}[1]{\mathcal{#1}}
\newcommand{\mb}[1]{\mathbf{#1}}
\newcommand{\ip}[2]{\left \langle{#1},{#2} \right \rangle}
\newcommand{\norm}[1]{\left \|{#1} \right \|}

\DeclareMathOperator{\Span}{\mathrm{Span}}

\DeclareMathOperator{\cNC}{\widetilde{\mathrm{NC}}}

\newcommand{\br}{\medskip\noindent}

\title{Some Fock spaces with depth two action}
\author{Michael Anshelevich, Jacob Mashburn}
\thanks{This work was supported in part by a Simons Foundation Collaboration Grant.}
\address{Department of Mathematics, Texas A\&M University, College Station, TX 77843-3368}
\email{manshel@math.tamu.edu, mashburn@math.tamu.edu}
\date{\today}

\begin{document}

\begin{abstract}
The subject of this paper are operators represented on a Fock space which act only on the two leading components of the tensor. We unify the constructions from \cite{AnsFree-Meixner,Lenczewski-Salapata-Kesten,Bozejko-Lytvynov-Meixner-I,Bozejko-Lytvynov-Meixner-II}, and extend a number of results from these papers to our more general setting. The results include the quadratic relation satisfied by (the kernel of) the free cumulant generating function, the resolvent form of the generating function for the Wick polynomials, and classification results for the case when the vacuum state on the operator algebra is tracial. We handle the generating functions in infinitely many variables by considering their matrix-valued versions.
\end{abstract}

\maketitle

%\tableofcontents

%\bigskip

\section{Introduction}

Symmetric, anti-symmetric, and full Fock spaces have been studied for almost a century. They have appeared in quantum mechanics and the theory of operator algebras; in combinatorics, representation theory, and the study of orthogonal polynomials; and in other fields. Over the last few decades, there has been a proliferation of other Fock spaces, with accompanying structures which to a greater or lesser degree share the properties of the classical ones. Such structures include commutation relations; Gaussian-type and Poisson-type operators; Wick products and orthogonal polynomials; It\^o calculus, etc. Among these numerous constructions, we will highlight three, all of which appeared roughly at the same time a decade ago. While quite different, they all share a common feature, which one might call ``depth two action'' or ``nearest neighbor coupling''. In this paper we study a more general construction extending the three above, which retains this feature.

To be specific, let $\mc{B}$ be a unital $\ast$-algebra, with a positive faithful linear functional $\phi$, and form the algebraic Fock space $\mc{F}_{alg}(\mc{B}) = \bigoplus_{n=0}^\infty \mc{B}^{\otimes n}$. For each $b \in \mc{B}$, on this space we will define operators $a^+(b), a^-(b), a^0(b)$ and $X(b) = a^+(b) + a^0(b) + a^-(b)$. Here the creation operator $a^+(b)$ is defined in the usual way, but the annihilation operator acts on simple tensors as
\[
a^-(b) (u_1 \otimes \ldots \otimes u_n) = (\gamma + \phi)[b u_1] u_2 \otimes \ldots \otimes u_n,
\]
where $\gamma + \phi \mathbf{1}_{\mc{B}}$ is some completely positive map. Note that this operator couples together the first two components of the tensor. This should be compared with, on the one hand, the free Fock space, where the annihilation operator acts only on the first component of the tensor; and, on the other hand, with the $q$-Fock space, where its action involves all components of the tensor. Note also that if $\gamma = 0$, this is just the standard free annihilation operator on the full Fock space $\mc{F}(L^2(\mc{B}, \phi))$. The preservation operator has the form
\[
a^0(b) (u_1 \otimes \ldots \otimes u_n) = \Lambda(b) u_1 \otimes u_2 \otimes \ldots \otimes u_n
\]
or more generally
\[
a^0(b) (u_1 \otimes \ldots \otimes u_n) = \Lambda(b \otimes u_1) \otimes u_2 \otimes \ldots \otimes u_n
\]
for a map $\Lambda$ satisfying a symmetry condition, and acts only on the first component of the tensor.

Before proceeding further, we now outline the three earlier constructions mentioned above. We only list some of the results about these constructions; each of the papers below contains other important results not pertinent to this introduction. It can easily be seen that these constructions all fit in the setting of the preceding paragraph.

In \cite{Lenczewski-Salapata-Kesten}, Lenczewski and Sa{\l}apata constructed a deformation of the free Fock space of $L^2(\mf{R}_+)$ with the annihilation operator of the form
\[
a^-(f) g(x_1, \ldots, x_{n}) = \int f(x_1) w(x_1, x_2) g(x_1, \ldots, x_{n}) \,dx_1
\]
for a specific function $w(s,t)$ of two variables. They studied combinatorial properties of operators $X(f)$ on this space, in particular showing that they have the Kesten distribution. In \cite{Bozejko-Lytvynov-Meixner-I,Bozejko-Lytvynov-Meixner-II}, Bo{\.z}ejko and Lytvynov constructed, in the setting of standard triples, a Fock space based on the annihilation operator
\[
a^-(f) g(x_1, \ldots, x_{n}) = \int f(x_1) \eta(x_1) g(x_1, \ldots, x_{n}) \,dx_1
\]
for some function $\eta$ on $\mf{R}$. They studied general (L{\'e}vy-type) processes on such a Fock space. Perhaps most importantly, they gave a characterization among these, by a certain continuity property, of what is natural to call the \emph{Free Meixner class}. The paper also contains several formulas for free cumulants of the distributions of such processes. In the second paper, they constructed a generating function for a system of polynomials (in infinitely many variables) associated to such a process in terms of certain operator-valued functions. Finally, in \cite{AnsFree-Meixner}, the first author studied the free Meixner states. Here the Fock space was based on a finite-dimensional space $\mf{R}^d$, and the annihilation operator is
\[
a^-_i (e_{u(1)} \otimes \ldots \otimes e_{u(n)}) = C_{u(1), u(2)} \ip{e_i}{e_{u(1)}} e_{u(2)} \otimes \ldots e_{u(n)}.
\]
The key results involved equations satisfied by the free cumulant generating function and the generating function for the Wick products (which can be identified with monic orthogonal polynomials in $d$ variables), as well as the constructions in which the vacuum state is tracial on the algebra of these polynomials.

Motivated by these three paper, our primary interest in this article is in
\begin{itemize}
\item
The joint distributions of $\set{X(u) : u \in \mc{B}}$, typically expressed in terms of appropriate cumulants.
\item
The Wick polynomials, which are polynomials in $\set{X(u) : u \in \mc{B}}$ such that \\
$W(u_1 \otimes u_2 \otimes \ldots \otimes u_n) \Omega = u_1 \otimes u_2 \otimes \ldots \otimes u_n$.
\item
The algebras generated by $\set{X(u) : u \in \mc{B}}$.
\end{itemize}

Free and Boolean cumulant generating functions for free Meixner families \cite{Boz-Bryc,AnsFree-Meixner} satisfy second-order equations. Depth two action on the Fock space results in such equations being satisfied not by the scalar-valued free cumulant generating function $R(u)$ itself, but by its $\mc{B}$-valued ``kernel'' $R'(u)$ with $R(u) = \phi[u R'(u) u]$. In fact,
\begin{equation}
\label{Eq:Cumulant-GF-Introduction}
	R'(u) v = v + R'(u) \gamma\left[ u R'(u) u \right] v + R'(u) \Lambda(u \otimes v).
	\end{equation}
This result easily generalizes to a finite family of variables $\set{u_i}_{i=1}^d$. To make sense of a generating function for joint free cumulants of infinitely many variables $\set{u_i}_{i=1}^\infty$, we take an approach different from \cite{Bozejko-Lytvynov-Meixner-II}. We form an infinite matrix which contains all the information about joint free cumulants of $\set{u_i}_{i=1}^\infty$, and still satisfies (an appropriate version of) equation~\eqref{Eq:Cumulant-GF-Introduction}. In the case when $\Lambda(u \otimes v) = \Lambda(u) v$, for $\set{u_i}_{i=1}^\infty$ uniformly small, this matrix corresponds to a genuine bounded operator. The analysis is similar in style to, but different from, computations with fully matricial free cumulants \cite{Popa-Vinnikov-NC-functions}.

We perform a similar analysis for the joint generating function of Wick polynomials. It can be interpreted as an infinite matrix, and under appropriate assumptions as a bounded operator. As in \cite{AnsFree-Meixner,Bozejko-Lytvynov-Meixner-II}, it has a resolvent-type form
\[
W(u) = (B(u) - X(u))^{-1}(B(u) - \phi[u^2]),
\]
where $B(u) = 1 + \Lambda(u) + (\gamma + \phi)[u^2]$. See Section~\ref{Section:GF}.

Finally, we investigate the situation when the vacuum state on the algebra generated by all operators $\set{X(u) : u \in \mc{B}}$ is tracial. The depth two nature of the action allows us to write down explicit conditions on $\phi, \gamma, \Lambda$ which guarantee this. In the case when $\gamma = 0$, the Fock space is the full Fock space, but the circular operators $X(u)$ are deformed by a non-trivial $\Lambda$. We show that one can always use $\Lambda$ to define a new multiplication on $\mc{B}$, so that the representation splits into a semicircular and a free compound Poisson parts. More generally, if $\gamma[u] = \eta u$ for $\eta$ central (related to the construction from \cite{Bozejko-Lytvynov-Meixner-I}) then one has a similar decomposition, but with the third component on which $\Lambda(u \otimes v) = \lambda u v$ for $\lambda$ central.

The paper \cite{AnsFree-Meixner} contained another example of a ``depth two action'' algebra with a tracial vacuum state, which naturally corresponded to the free multinomial distribution. This example also generalizes to our setting here. It is described in a forthcoming paper \cite{Ans-Mashburn-Centered}.

%For completeness, we note that other approaches to the notion of quadratic Fock spaces, mostly unrelated to ours, have been considered in \cite{Sniady-SWN,Accardi-Dhahri-Quadratic,Ans-Wick-products}.

The paper is organized as follows. After the Introduction, in Section~\ref{Section:Construction} we present the main construction. In Section~\ref{Section:Cumulants}, we prove formulas for joint moments, and Boolean and free cumulants, of the operators $\set{X(u) : u \in \mc{B}}$. We also compare and contrast these formulas with the operator-valued results of \cite{Ans-Williams-Jacobi}. In particular, unlike in \cite{Ans-Williams-Jacobi}, the inner product in this paper is scalar-valued rather than $\mc{B}$-valued. In Section~\ref{Section:GF}, we discuss Wick polynomials, and matricial generating functions for them and for the free cumulants. In Section~\ref{Section:Norms}, we provide conditions under which operators $X(u)$, as well as various generating functions, are bounded. In Section~\ref{Section:Trace} we derive the conditions for the vacuum state to be tracial. The results in Sections~\ref{Section:GF} and \ref{Section:Norms} are proven in the setting of $\Lambda(u \otimes v) = \Lambda(u) v$; in contrast, the results in Section~\ref{Section:Trace} are of main interest for general $\Lambda$. Finally, in Section~\ref{Sec:Examples}, we describe in detail the examples which motivated this paper, as well as their generalizations covered by our construction. We also prove a representation theorem under the assumption that the vacuum state is tracial.

\section{The construction}
\label{Section:Construction}

\begin{Defn}
Let $\mc{B}$ and $\mc{C}$ be $\ast$-algebras. An element $b \in \mc{B}$ is positive if $b = \sum_{i=1}^k u_i^\ast u_i$ for some $k$ and $u_i \in \mc{B}$. A map $T: \mc{B} \rightarrow \mc{C}$ is positive if for each $u \in \ \mc{B}$, $T(u^\ast u)$ is positive in $\mc{C}$. It is faithful if $T(u^\ast u) = 0$ only for $u = 0$. It is completely positive if for each $n$, the map
\begin{equation*}
T_n : M_n(\mc{B}) \rightarrow M_n(\mc{C}), \,\,\, T_n\left( [a_{ij}]_{i,j=1}^n \right) = [T(a_{ij})]_{i,j=1}^n
\end{equation*}
is positive, where we use the usual $\ast$-structure on $M_n(\mc{B})$.
\end{Defn}

\begin{Construction}
\label{Construction:gamma}
Let $\mc{B}$ be a unital *-algebra, equipped with star-linear maps $\phi : \mc{B} \rightarrow \mf{C}$, $\gamma : \mc{B} \rightarrow \mc{B}$, and $\Lambda: \mc{B} \otimes_{\text{alg}} \mc{B} \rightarrow \mc{B}$ such that $\phi$ is positive and faithful, $\gamma + \phi = \gamma + \phi \mb{1}_{\mc{B}}$ is completely positive (semi-definite), and $\Lambda$ satisfies
\begin{equation}
\label{Eq:a0-symmetric}
\phi[v^\ast \Lambda(b \otimes u)] = \phi[\Lambda(b^\ast \otimes v)^\ast u], \quad \gamma[v^\ast \Lambda(b \otimes u)] = \gamma[\Lambda(b^\ast \otimes v)^\ast u],
\end{equation}
On the algebraic Fock space,
\[
\mc{F}_{alg}(\mc{B}) = \mf{C} \Omega \oplus \bigoplus_{n=1}^\infty \mc{B}^{\otimes n},
\]
define the inner product by the linear extension of
\begin{multline}
\label{Eq:Inner-product}
\ip{u_1 \otimes \ldots \otimes u_n}{v_1 \otimes \ldots \otimes v_k}_{\gamma,\phi} \\
= \delta_{n=k} \phi \left[ v_n^\ast (\gamma + \phi)[v_{n-1}^\ast (\gamma + \phi)[ \ldots (\gamma + \phi)[ v_1^\ast u_1] \ldots ] u_{n-1}] u_n \right].
\end{multline}
This inner product is positive semi-definite but not, in general, positive definite. Denote by $\mc{F}_{\gamma, \phi}(\mc{B})$ the completion of the quotient  $\mc{F}_{alg}(\mc{B}) / \mc{N}$ by the subspace $\mc{N}$ of elements of zero seminorm. Next, for each $b \in \mc{B}$, consider operators on $\mc{F}_{alg}(\mc{B})$
\[
a^+(b) (u_1 \otimes \ldots \otimes u_n) = b \otimes u_1 \otimes \ldots \otimes u_n,
\]
\[
a^-(b) (u_1 \otimes \ldots \otimes u_n) = (\gamma + \phi)[b u_1] u_2 \otimes \ldots \otimes u_n,
\]
\[
a^-(b) (u_1) = \phi[b u_1] \Omega,
\]
\[
a^0(b) (u_1 \otimes \ldots \otimes u_n) = \Lambda(b \otimes u_1) \otimes  u_2 \otimes \ldots \otimes u_n,
\]
\[
a^-(b) (\Omega) = a^0(b)(\Omega) = 0,
\]
and
\[
X(b) = a^+(b) + a^-(b) + a^0(b).
\]
Several of the results below address the questions of when these operators are defined on $\mc{F}_{alg}(\mc{B}) / \mc{N}$ or on $\mc{F}_{\gamma, \phi}(\mc{B})$.

Denote
\[
\Gamma^{alg}_{\gamma, \Lambda}(\mc{B}, \phi) = \mathrm{Alg}_{\mf{C}}(X(b) : b \in \mc{B}) = \mathrm{Alg}_{\mf{C}}(X(b) : b \in \mc{B}^{sa})
\]
and define on it the \emph{vacuum state} $A \mapsto \ip{A \Omega}{\Omega}$.
\end{Construction}

\begin{Remark}
The following observations are straightforward.
\begin{itemize}
\item
If for some $t < 1$, $\gamma + t \phi$ is still completely positive, then the inner product~\eqref{Eq:Inner-product} is positive definite.
\item
%\label{Lemma:Bounded-semi-norm}
If the operators in Construction~\ref{Construction:gamma} are bounded with respect to the semi-norm $\norm{\cdot}_{\gamma, \phi}$ over $\mc{F}_{alg}(\mc{B})$, then they are well-defined over $\mc{F}_{\gamma, \phi}(\mc{B})$ and are also norm-bounded.
\item
For $\vec{\zeta}, \vec{\xi} \in \mc{F}_{alg}(\mc{B})$,
\[
\ip{a^+(b) \vec{\zeta}}{\vec{\xi}}_{\gamma, \phi} = \ip{\vec{\zeta}}{a^-(b^\ast) \vec{\xi}}_{\gamma, \phi}, \quad
\ip{a^0(b) \vec{\zeta}}{\vec{\xi}}_{\gamma, \phi} = \ip{\vec{\zeta}}{a^0(b^\ast) \vec{\xi}}_{\gamma, \phi}.
\]
In particular, for self-adjoint $b$, $X(b)$ is symmetric. In the case when $\Lambda(b \otimes u) = \Lambda(b) u$ for $\Lambda : \mc{B} \rightarrow \mc{B}$, assumption~\eqref{Eq:a0-symmetric} simplifies to $(\Lambda(b))^\ast = \Lambda(b^\ast)$.
\end{itemize}
\end{Remark}

\section{Moments, free cumulants, Boolean cumulants}
\label{Section:Cumulants}
In this section, we give combinatorial expressions for the moments and free cumulants of the variables $\{X(u_i)\}$. For any family of bounded operators $X_1, ..., X_n$ over a Hilbert space $\mc{H}$ and a state $\psi$ over $B(\mc{H})$, their \textit{moments} are the numbers
\begin{equation*}
\psi[X_{i(1)}...X_{i(k)}].
\end{equation*}

\paragraph{\textbf{Combinatorial background. }} A \textit{partition} $\pi$ of a subset $S\subset \mathbb{N}$ is a collection of disjoint subsets of $S$ (called $blocks$ of $\pi$) whose union equals $S$. We will use $i \overset{\pi}{\thicksim} j$ to say that $i$ and $j$ are in the same block of $\pi$. In this paper, we will only be concerned with partitions of $[n] :=\{1, 2, ..., n\}$.

Let NC$(n)$ denote the set of \textit{noncrossing partitions} over $[n]$, that is, those partitions $\pi$ such that there are no $i < j < k < \ell$ such that $i \overset{\pi}{\thicksim} k$ and $j \overset{\pi}{\thicksim} \ell$ unless all four are in the same block.
If $i\in[n]$ is the first element of its block, we will call it an \textit{opening element}, while the last of its block will be called a \textit{closing element}. If $i$ is neither, it will be called a \textit{middle element}.

For distinct blocks $V$ and $W$ of a noncrossing partition $\pi$, $V$ is said to be \textit{inner with respect to} $W$ if $o_W < o_V < c_V < c_W$, where $o_V$ and $o_W$ are the opening elements of $V$ and $W$, respectively, while $c_V$ and $c_W$ are their closing elements. We will simply say a block is \textit{inner} if it is inner with respect to some block, and \textit{outer} if it is not.

Finally, let Int$(n)$ denote the \textit{interval partitions} over $[n]$, that is, those partitions $\pi$ such that whenever $i<j$ and $i \overset{\pi}{\thicksim} j$, we have $i \overset{\pi}{\thicksim} k$ for all $i < k < j$.

The \textit{free cumulants} of $X_1, ..., X_n$ are defined inductively via the \textit{moment-cumulant formula}
\begin{equation*}
\psi[X_{i(1)}, ..., X_{i(k)}] = \sum_{\pi\in\text{NC}(n)} R_\pi [X_{i(1)}, ..., X_{i(k)}],
\end{equation*}
where
\begin{equation*}
R_\pi [X_1, ..., X_k] = \prod_{V\in\pi} R[X_{V(1)}, ..., X_{V(|V|)}].
\end{equation*}
Their \textit{Boolean cumulants} are defined in almost exactly the same manner:
\begin{equation*}
\psi[X_{i(1)}, ..., X_{i(k)}] = \sum_{\pi\in\text{Int}(n)} B_\pi [X_{i(1)}, ..., X_{i(k)}],
\end{equation*}
where
\begin{equation*}
B_\pi [X_1, ..., X_k] = \prod_{V\in\pi} B[X_{V(1)}, ..., X_{V(|V|)}].
\end{equation*}

\begin{Notation}
Denote $\mathrm{NC}_{ns}(n)$ the set of noncrossing partitions of $[n]$ with no singleton blocks. For $u_1, ..., u_n \in\mathcal{B}$, in the moment and Boolean cumulant formulas below,  we will assign to each $\pi\in \mathrm{NC}_{ns}(n)$ the weight operator on $\mathcal{F}_{alg}(\mathcal{B})$
	\begin{equation*}
	W_M(\pi) = \prod_{i=1}^n a_i (u_i),
	\end{equation*}
	where
	\begin{equation*}
	a_i = \begin{cases}
	a^+, \text{ if }i\text{ is a closing element,} \\
	a^-, \text{ if }i\text{ is an opening element,} \\
	a^0, \text{ if }i\text{ is a middle element.}
	\end{cases}
	\end{equation*}
	
\end{Notation}

	Then we have the following standard expansion.

\begin{Prop}
\label{Prop:moments}
Given $u_1, ..., u_n\in\mathcal{B}$, we have the following mixed moment formula:
\begin{equation}
\ip{X(u_1)...X(u_n) \Omega}{\Omega} = \sum_{\pi\in\mathrm{NC}_{ns}(n)} \langle W_M(\pi) \Omega,\Omega\rangle_{\gamma,\phi},
\end{equation}
if $n \geq 2$ and zero if $n=1$.
\end{Prop}

\begin{Notation}
	Let $\widetilde{\mathrm{NC}}(n) = \{\pi\in\mathrm{NC}(n) | 1 \overset{\pi}{\thicksim} n \}$, and similarly, let
\[
\widetilde{\mathrm{NC}}_{ns}(n) = \{\pi\in\mathrm{NC}(n) | 1 \overset{\pi}{\thicksim} n \text{ and }\pi \text{ has no singleton blocks} \}.
\]
\end{Notation}

\begin{Lemma}
\label{Lemma:Boolean}
For $n\geq 2$, given $u_1, ..., u_n\in\mathcal{B}$, we have the following mixed Boolean cumulant formula:
\begin{equation}
B_n[X(u_1), ..., X(u_n)] = \sum_{\pi\in \widetilde{\mathrm{NC}}_{ns}(n)} \langle W_M(\pi)\Omega, \Omega \rangle_{\gamma,\phi},
\end{equation}
and for $n=1$, the cumulant is zero.
\end{Lemma}

\begin{proof}
$n=1$ is clear. Assume the result holds for all natural numbers less than some $n$, and take $u_1, ..., u_n\in\mathcal{B}$. By Proposition \ref{Prop:moments}, we have
\begin{equation*}
\sum_{\pi\in\mathrm{Int}(n)} B_\pi [X(u_1),...,X(u_n)] = \ip{X(u_1)...X(u_n) \Omega}{\Omega} = \sum_{\pi\in\mathrm{NC}_{ns}(n)} \langle W_M(\pi)\Omega, \Omega\rangle_{\gamma,\phi}.
\end{equation*}

For convenience, denote by $\mathrm{NC}_{ns, mo}(n)$ the noncrossing, no-singleton partitions of $[n]$ that have more than one outer block. After isolating the $n$th cumulant (corresponding to the partition $\hat{1}_n$ consisting of a single block), we get
\begin{eqnarray*}
B_{n}[X(u_1),...,X(u_n)] &=& \sum_{\pi\in\mathrm{NC}_{ns}(n)} \langle W_M(\pi)\Omega, \Omega \rangle_{\gamma,\phi} - \sum_{\pi\in\mathrm{Int}(n)\backslash\{\hat{1}_n\}} B_\pi [X(u_1),...,X(u_n)] \\
&=& \sum_{\pi\in\mathrm{NC}_{ns}(n)} \langle W_M(\pi)\Omega, \Omega\rangle_{\gamma,\phi} - \sum_{\pi\in\mathrm{Int}(n)\backslash\{\hat{1}_n\}} \prod_{V\in\pi} \sum_{\sigma\in \widetilde{\mathrm{NC}}_{ns}(|V|)} \langle W_M(\sigma)\Omega , \Omega \rangle_{\gamma,\phi} \\
&=& \sum_{\pi\in\mathrm{NC}_{ns}(n)} \langle W_M(\pi)\Omega, \Omega\rangle_{\gamma,\phi} - \sum_{\sigma\in \mathrm{NC}_{ns,mo}(n)} \langle W_M(\pi)\Omega , \Omega\rangle_{\gamma,\phi} \\
&=& \sum_{\pi\in\widetilde{\mathrm{NC}}_{ns}(n)} \langle W_M(\pi)\Omega, \Omega\rangle_{\gamma,\phi},
\end{eqnarray*}
where the second equality follows from the induction hypothesis.
\end{proof}

\begin{Notation}
Define the operator $a_\gamma^\sim(b)$ ($b\in\mathcal{B}$) by linear extension of
	\begin{equation*}
	a_\gamma^\sim (b) (u_1 \otimes \ldots \otimes u_n) = \gamma[b u_1] u_2 \otimes \ldots \otimes u_n \text{  for }n\geq 2,
	\end{equation*}
	\begin{equation*}
	a_\gamma^\sim (b) (u_1) = 0 \text{  for }n=1,\text{ and}
	\end{equation*}
	\begin{equation*}
	a_\gamma^\sim (b) (\Omega) = 0.
	\end{equation*}
	
The operator $a_\phi^\sim (b)$ is defined in a similar manner (for $n\geq 2$, apply $\phi$, and $a_\phi^\sim (b) (\Omega) = 0$), with one critical exception:
	\begin{equation*}
	a_\phi^\sim (b) (u_1) = \phi [b u_1]\Omega \text{ for }n=1.
	\end{equation*}
Thus $a^-(b) = a_\gamma^\sim (b) + a_\phi^\sim (b)$.

For $u_1, ..., u_n \in\mathcal{B}$, in the free cumulant formula below, we will assign to each $\pi\in \mathrm{NC}(n)$ the weight operator on $\mathcal{F}_{alg}(\mathcal{B})$
	\begin{equation*}
	W_C(\pi) = \prod_{i=1}^n a_i (u_i),
	\end{equation*}
	where
	\begin{equation*}
	a_i = \begin{cases}
	a^+, \text{ if }i\text{ is a closing element,} \\
	a_\gamma^\sim, \text{ if }i\neq 1\text{ and is an opening element,} \\
	a_\phi^\sim, \text{ if }i = 1 \text{, or} \\
	a^0, \text{ if }i\text{ is a middle element.}
	\end{cases}
	\end{equation*}
\end{Notation}

\begin{Prop}\label{Prop:FreeCumulants}
	For $n\geq 2$, given $u_1, ..., u_n\in\mathcal{B}$, we have the following mixed free cumulant formula:
	\begin{equation}
	R_n[X(u_1), ..., X(u_n)] = \sum_{\pi\in\widetilde{\mathrm{NC}}_{ns}(n)} \langle W_C(\pi)\Omega, \Omega\rangle_{\gamma,\phi},
	\end{equation}
	and for $n=1$, the cumulant is zero.
\end{Prop}

\begin{proof}
By Theorem~1 in \cite{Belinschi-Nica-Eta} and Lemma~\ref{Lemma:Boolean},
\begin{equation}
\label{Eq:FCProofFirstEq}
\begin{split}
\sum_{\pi\in\widetilde{\mathrm{NC}}_{ns}(n)} R_\pi [X(u_1),...,X(u_n)]
& = B_n [X(u_1),...,X(u_n)] \\
& = \sum_{\pi\in\widetilde{\mathrm{NC}}_{ns}(n)} \sum_{\ell\in\text{Block labelings}(\gamma,\phi)} \langle  W(\pi,\ell)\Omega , \Omega\rangle_{\gamma,\phi},
\end{split}
\end{equation}
where $\ell$ is a labeling of the blocks of $\pi$ with either a $\gamma$ or $\phi$ such that the block containing 1 and $n$ is labeled $\phi$, and $W(\pi,\ell) = \prod_{i=1}^n a_i(u_i)$, where
\begin{equation*}
a_i = \begin{cases}
a^+, \text{ if }i\text{ is a closing element,} \\
a_\gamma^\sim, \text{ if }i\text{ is an opening element of a block labeled }\gamma , \\
a_\phi^\sim, \text{ if }i\text{ is an opening element of a block labeled }\phi , \\
a_0, \text{ if }i\text{ is a middle element.}
\end{cases}
\end{equation*}
This splitting of terms follows from linearity and the fact that the opening of each block is weighted with $(\gamma + \phi)$ of a product involving the operator in that position.

Since the first moment (and thus free cumulant) is zero, for $n=2$, we have
\[R_2[X(u_1),X(u_2)] = M_2[X(u_1)X(u_2)] = \langle a^-(u_1)a^+(u_2)\Omega , \Omega\rangle_{\gamma,\phi},\] and for $n=3$, we have
\[R_3[X(u_1),X(u_2),X(u_3)] = M_3[X(u_1)X(u_2)X(u_3)] = \langle a^-(u_1)a^0(u_2)a^+(u_3)\Omega , \Omega\rangle_{\gamma,\phi}.\]

For induction, assume for all natural numbers less than some $n$, the formula holds.

Let $\ll$ denote the partial order on NC$(n)$ defined by
\begin{equation*}
\sigma \ll \pi \,\,\,\,\, \Leftrightarrow \,\,\,\,\, \sigma \leq \pi \,\text{ and }\, \forall\, V\in\pi , \,\, \min(V) \overset{\sigma}{\thicksim} \max(V).
\end{equation*}

Then the inductive hypothesis implies the left-hand side of (\ref{Eq:FCProofFirstEq}) equals
\begin{equation*}
R_n[X(u_1),...,X(u_n)] + \sum_{\pi\in\widetilde{\mathrm{NC}}_{ns}(n) \backslash \{1_n\}} \sum_{\sigma \ll \pi} \langle W(\sigma)\Omega , \Omega \rangle_{\gamma,\phi},
\end{equation*}
where each $\sigma$ is labeled such that for all $V\in\pi$, $\sigma |_V$ has its unique outer block labeled $\phi$, while the rest (that is, the inner blocks) are labeled $\gamma$.

By Remark~2.14 in \cite{Belinschi-Nica-Eta},
\[
\{\pi\in\mathrm{NC}(n) : \sigma \ll \pi\} \simeq \{S\subset \sigma : S\text{ contains the outer blocks of } \sigma\} =: S_\sigma.
\]
With this bijection, the above sum can be rewritten as
\begin{equation*}
R_n[X(u_1),...,X(u_n)] + \sum_{\sigma\in\widetilde{\mathrm{NC}}_{ns}(n) \backslash \{1_n\}} \sum_{S\in S_\sigma} \langle W(\sigma,\ell_S)\Omega , \Omega \rangle_{\gamma,\phi},
\end{equation*}
where $\ell_S$ is the labeling constructed by giving blocks in $S$ the label $\phi$ and the rest $\gamma$. Subtracting the sum from both sides gives the result.
\end{proof}

\begin{Defn}
	\label{Defn:R-prime}
	For each $u_1, \ldots, u_k \in \mc{B}$, define the linear operator $R'[u_1,...,u_k]$ on $\mc{B}$ by
	\begin{equation}
\label{Eq:R'-definition2}
	R[X(u_1), \ldots, X(u_n)] = \ip{R'[u_2, \ldots, u_{n-1}] u_n}{u_1^\ast} = \phi[u_1 R'[u_2, \ldots, u_{n-1}] u_n] .
	\end{equation}
In general, this operator need not be bounded.
\end{Defn}

\begin{Lemma}
	\label{Lemma:R-prime}
We have the expansion
	\begin{equation}
\label{Eq:R'-definition}
	R'[u_1,...,u_k] = \left( \sum_{\pi \in \mathrm{Int}(\{1,...,k\})} \prod_{V\in\pi} w(V) \right),
	\end{equation}
	where the products are ordered by each block's appearance in the partition (from left to right), and the weights are given by $w(\{i\}) = a^0 (u_i)$ and $w(\{i_1,...,i_k\}) = \gamma[u_{i_1}R'[u_{i_2},...,u_{i_{k-1}}]u_{i_k}]$ for $k\geq 2$, with $R'[\emptyset] = 1$.

In particular, for $\Lambda(u \otimes v) = \Lambda(u) v$, $R'[u_1,...,u_k]$ is the operator of multiplication by an element of $\mc{B}$.
\end{Lemma}

\begin{proof}
	To show that $\ip{\left( \sum_{\pi \in \mathrm{Int}(\{2,...,n-1\})} \prod_{V\in\pi} w(V) \right) u_n }{u_1^\ast} = \sum_{\pi\in\widetilde{\mathrm{NC}}_{ns}(n)} \langle W(\pi)\Omega , \Omega \rangle_{\gamma,\phi}$, we will prove a slightly stronger statement:
	\begin{equation}
	\label{Eq:OperatorCumulants}
	R'' :=	u_1 \left( \sum_{\pi \in \mathrm{Int}(\{2,...,n-1\})} \prod_{V\in\pi} w(V) \right) u_n = \sum_{\pi\in\widetilde{\mathrm{NC}}_{ns}(n)} W'(\pi)\Omega ,
	\end{equation}
	where the weights in the right-hand side are the same as in the cumulant formula, except all opening steps at 1 will be weighted by the identity map instead of $a_\phi^\sim$.
	
	For $n=2$, we have $u_1u_2$ on both sides. For $n=3$, we have $u_1a^0(u_2)u_3$. For induction, assume the lemma holds for all natural numbers less than some $n$. Then
	\begin{equation*}
	\sum_{\pi\in\widetilde{\mathrm{NC}}_{ns}(n)} W'(\pi)\Omega = u_1 \left( \sum_{\pi \in \mathrm{Int}(\{2,...,n-1\})} \prod_{V\in\pi} w(V) \right) u_n ,
	\end{equation*}
	where for $k\geq 2$, $w(\{i_1,...,i_k\}) = \gamma\left[u_{i_1} \left( \sum_{\pi \in \mathrm{Int}(\{2,...,k-1\})} \prod_{V\in\pi} w(V) \right) u_{i_k}\right] $.
	
	At this point, we are done, since each $\pi\in \widetilde{\mathrm{NC}}_{ns}(n)$ can be uniquely constructed by taking some $\sigma\in \mathrm{Int}\{2, ..., n-1\}$, then constructing the unique outer block $\{1, \text{singletons}(\sigma),n\}$, then each nested block immediately below the outer block is recursively constructed in the same manner.
	
	Finally, applying the vacuum state to both sides of (\ref{Eq:OperatorCumulants}) gives the result.
\end{proof}

Next, we prove an equation which characterizes the generating function for this family of operators, which in turn will give us an equation for the free cumulant generating function.

\begin{Lemma}
\label{Lemma:R'-recursion}
Denote $R'_n[u] := R'[u,..., u]$ ($n$ arguments), where $R'_0 [u] = 1$. Then
\begin{equation}
\label{Eq:R'-recursion}
R'_n[u] = \sum_{i=0}^{n-2} R'_i[u]\gamma [u R'_{n-i-2}[u]u] + R'_{n-1}[u]a^0(u).
\end{equation}
\end{Lemma}

\begin{proof}
The claim is analogous to the recursion for the number $I_n$ of interval partitions of length $n$,
\begin{equation*}
I_n = \sum_{i=0}^{n-1} I_i,
\end{equation*}
where $n-i$ is the number of elements in the block containing $n$. The right-hand side in \eqref{Eq:R'-recursion} is obtained in a similar manner, in which each term is obtained by collecting all terms in the sum \eqref{Eq:R'-definition} (over interval partitions) for $R'_n[u]$ for which the weight for the block containing $n$ is a factor. Thus, by summing over $i$ where $n-i$ is the number of elements in the block containing $n$, the term corresponding to each $i$ is a product of $R'_i[u]$ and either $\gamma [u R'_{n-i-2}[u]u]$, the weight of the block containing $n$ for $i \leq n-2$, or $a^0(u)$, the weight of the singleton block containing $n$ for $i = n-1$.
\end{proof}

\begin{Thm}\label{Thm:CumulantGF}
	Let $R'(u) = 1 + \sum_{n=1}^\infty R'_n[u]$ be the generating function of $R'_n[u]$, $n \geq 0$. Then for any $v \in \mc{B}$,
\begin{equation}
	R'(u) v = v + R'(u) \gamma\left[ u R'(u) u \right] v + R'(u) \Lambda(u \otimes v).
	\end{equation}
\end{Thm}

\begin{proof}
Apply Lemma~\ref{Lemma:R'-recursion}.
\end{proof}

\begin{Example}
For $\Lambda(u \otimes v) = \Lambda(u) v$, $R'(u)$ satisfies
\[
R'(u) = \left( 1 - \Lambda(u) - \gamma[u R'(u) u] \right)^{-1},
\]
and so can be expanded into a formal continued fraction.
\end{Example}

\begin{Cor}\label{Cor:MultiCumulantGF}More generally, for $u_1, ..., u_k\in\mc{B}$, let
\[
R'(u_1, ..., u_k) = \sum_{n=0}^\infty \sum_{j : [n] \rightarrow [k]} R'_n[u_{j(1)},...,u_{j(n)}]
\]
be the generating function for the mixed $R'_n$. Then
	\begin{multline}
	R'(u_1, ..., u_k)v \\
= v + R'(u_1, ..., u_k)\gamma\left[ \left(\sum_{i=1}^{k}u_i\right) R'(u_1, ..., u_k)\left(\sum_{j=1}^{k}u_j\right) \right]v + R'(u_1, ..., u_k) \sum_{i=1}^{k}\Lambda(u_i, v).
	\end{multline}
\end{Cor}
\begin{proof}
	Apply Theorem~\ref{Thm:CumulantGF} to $R'\left( \sum_{i=1}^{k}u_i \right) = R'(u_1, ..., u_k)$.
\end{proof}

\begin{Remark}
The generating function for $R_n''[u]$ is given by $R''(u) = uR'(u) u$, since $R_n''[u]$ is only defined for $n\geq 2$. Thus, $R''(u)$ satisfies
	\begin{equation}
	R''(u) = u^2 + u R'(u) \gamma\left[ R''(u) \right] u + u R'(u) \Lambda(u, u).
	\end{equation}
	Applying the vacuum state to both sides gives an equation for the cumulant generating function. Note that $R''(u)$ is a series of elements of $\mc{B}$.
\end{Remark}

\begin{Remark}\label{Rmk:Ans-Williams-Jacobi-comparison}
Constructions and results in this section are reminiscent of operator-valued probability theory, such as those in \cite{Ans-Williams-Jacobi}. In this remark we indicate how these constructions differ. The map
\[
b_1 \otimes b_1 \otimes \ldots \otimes b_n \mapsto b_1 X b_2 X \ldots b_n X, \quad \Omega \mapsto 1_{\mc{B}}
\]
is an isomorphism from the algebraic Fock space $\bigoplus_{n=0}^\infty \mc{B}^{\otimes n}$ onto a subspace of non-commutative polynomials $\mc{B} \langle X \rangle$. Using this identification, the relation between the operators in this article (for $\Lambda(f \otimes g) = \Lambda(f) g$) and in Proposition~3.1 from \cite{Ans-Williams-Jacobi} is:
\[
a^+(b) \leftrightarrow b a^\ast, \quad a^-(b) \leftrightarrow a \ b, \quad a^0(b) \leftrightarrow p
\]
with $\alpha_1 = \phi$, $\alpha_n = \gamma + \phi$ for $n \geq 2$, and $\lambda_n = \Lambda(b)$ for $n \geq 1$. Thus the operator $X(1)$ here is the same as the $\mc{B}$-valued $X$, but the interaction with the algebra $\mc{B}$ is different in the two settings. Similarly, the identity
\begin{equation*}
	R'(b) = 1 + R'(b)\gamma\left[ b R'(b)b \right] + R'(b) \Lambda(b).
	\end{equation*}
satisfied by the (under appropriate assumptions) $\mc{B}$-valued generating function $R'(b)$ (from Theorem~\ref{Thm:CumulantGF}) is similar to, but different from the relation
\[
b^{-1} R_\mu(b) b^{-1} = 1 + \gamma[R_\mu(b)b^{-1}]R_\mu(b)b^{-1} + \lambda R_\mu(b)b^{-1}.
\]
from Proposition~3.22 in \cite{Ans-Williams-Jacobi}. They do again coincide (up to a flip) for $b = 1$.
\end{Remark}

In light of the preceding remark, the following is related to Theorem 2 from \cite{AnsFree-Meixner}.

\begin{Prop}
Let $\mc{B}$ be a $C^\ast$-algebra and $u_1, \ldots, u_n \in \mc{B}^{sa}$. Assume that $\Lambda(u \otimes v) = \Lambda(u) v$. Let $X_1, \ldots, X_n$ be a $\mc{B}$-valued semicircular system with means $\Lambda(u_1), \ldots, \Lambda(u_n)$ and the covariance matrix $\eta: \mc{B} \rightarrow M_n(\mc{B})$ with $\eta_{ij}[b] = \gamma[u_i b u_j]$; the existence of such a system in some $\mc{B}$-valued non-commutative probability space $(\mc{A}, \mf{E}, \mc{B})$ is guaranteed  by \cite{Shl-A-semicircular}. Then
\[
R[X(u_{j(1)}), X(u_{j(2)}), \ldots, X(u_{j(k-1)}), X(u_{j(k)})] = \phi[u_{j(1)} \mf{E}[X_{j(2)} \ldots X_{j(k-1)}] u_{j(k)}].
\]
\end{Prop}

\section{Generating functions}
\label{Section:GF}

\subsection{Wick Polynomials}

\begin{Defn}
\label{Defn:Wick}
For $u_1 \otimes \ldots \otimes u_n \in \mc{B}^{\otimes n}$, define the operator $W(u_1 \otimes \ldots \otimes u_n)$ on the algebraic Fock space by
the recursion
\[
\begin{split}
& W(b \otimes u_1 \otimes \ldots \otimes u_n) \\
&\quad = X(b) W(u_1 \otimes \ldots \otimes u_n) - W(a^0(b) (u_1 \otimes \ldots \otimes u_n)) - W(a^-(b)(u_1 \otimes \ldots \otimes u_n)) \\
&\quad = X(b) W(u_1 \otimes \ldots \otimes u_n) - W(\Lambda(b \otimes u_1) \otimes u_2 \otimes \ldots \otimes u_n) - W((\gamma + \phi)[b u_1] u_2 \otimes \ldots \otimes u_n)
\end{split}
\]
with the initial conditions
\[
W(\emptyset) = I, \quad W(u_1) = X(u_1), \quad W(u_1 \otimes u_2) =  X(u_1) W(u_2) - W(\Lambda(u_1 \otimes u_2)) - \phi[u_1 u_2].
\]
It is easy to check using the recursion that
\[
W(u_1 \otimes \ldots \otimes u_n) \Omega = u_1 \otimes \ldots \otimes u_n.
\]
Since each $W(\vec{\xi}) \in \Gamma^{alg}_{\gamma, \Lambda}(\mc{B}, \phi)$, it follows that $\Omega$ is cyclic for this algebra. $W(\vec{\xi})$ is called a \emph{Wick polynomial}. In particular, we denote $W_n(u) = W(u^{\otimes n})$, and note that $W_n(u) \Omega = u^{\otimes n}$.
\end{Defn}

For the remainder of this section, we assume that $\Lambda$ has the special form $\Lambda(u \otimes v) = \Lambda(u) v$.

The next proposition states that formally, the generating function for univariate Wick polynomials is a rational function in $X(u)$. In Section \ref{Section:Norms}, we will give conditions under which the series defining the generating function is convergent and corresponds to a well defined operator.

\begin{Prop}
Denote $W(u) = 1 + \sum_{n=1}^\infty W_n(u)$ and
\[
b(u) = 1 + \Lambda(u) + (\gamma + \phi)[u^2].
\]
Then
\[
(b(u) - X(u)) W(u) = b(u) - \phi[u^2].
\]
\end{Prop}

\begin{proof}
\[
X(u) W_n(u) = W_{n+1}(u) + (\gamma + \phi)[u^2] W_{n-1}(u) + \Lambda(u) W_n(u),
\]
\[
X(u) W_1(u) = W_2(u) + \phi[u^2] + \Lambda(u) W_1(u),
\]
\[
X(u) = W_1(u).
\]
So
\[
X(u) W(u) = W(u) - 1 + \phi[u^2] + (\gamma + \phi)[u^2] (W(u) - 1) + \Lambda(u) (W(u) - 1). \qedhere
\]
\end{proof}

\subsection{Matricial generating functions}

In this section, we will explore a means of recovering the multi-variable polynomials $\{W(u_n, ..., u_m) : 1 \leq n \leq m < \infty \}$ for a sequence $\{u_i : i \in \mf{N}\}$ in $\mc{B}$. Rather than trying to form their generating function as a power series, we will organize these polynomials into an infinite matrix. In the next section, this matrix will be identified with a bounded operator. In the following proposition, $\mc{L}(\ell^2)$ denotes the algebra of linear (not necessarily bounded) operators, and $\set{E_{ij}}$ are the standard matrix units in it.

\begin{Prop}
\label{Prop:Infinite-W}
Let $\set{u_i : i \in \mf{N}} \subset \mc{B}$. Define matrices with operator entries $\Phi \in \mc{L}(\ell^2)$, $\Gamma, A^0 \in \mc{B} \otimes \mc{L}(\ell^2)$, $X, W \in \mc{L} (\mc{F}_{\gamma, \phi}(\mc{B}) \otimes \ell^2)$ as follows:
\begin{align*}
X &= \sum_{i=1}^\infty X(u_i) \otimes E_{i, i+1}, \\
W_n &= W(u_i\otimes \ldots\otimes u_{i+n-1}) E_{i, i+n} \\
W &= \sum_{n=1}^\infty W_n \\
&= \sum_{i=1}^\infty \left( E_{i, i} + \sum_{n=1}^\infty W(u_i \otimes \ldots \otimes u_{i+n-1}) \otimes E_{i, i + n} \right), \\
\Phi &= \sum_{i=1}^\infty \phi[u_i u_{i+1}] \otimes E_{i, i+2}, \\
\Gamma &= \sum_{i=1}^\infty \gamma[u_i u_{i+1}] \otimes E_{i, i+2}, \\
A^0 &= \sum_{i=1}^\infty \Lambda(u_i) \otimes E_{i, i+1}.
\end{align*}

Denote $B = I + A^0 + \Gamma + \Phi \in \mc{L} (\mc{F}_{\gamma, \phi}(\mc{B}) \otimes \ell^2)$. Then
\[
(B - X) W = (B - \Phi).
\]
\end{Prop}

\begin{proof}
For $k > i+2$,
\[
\begin{split}
X(u_i) W(u_{i+1} \otimes \ldots \otimes u_{k-1})
& = W(u_i \otimes \ldots \otimes  u_{k-1}) + \Lambda(u_i) W(u_{i+1} \otimes \ldots \otimes u_{k-1}) \\
&\quad + (\gamma[u_i u_{i+1}] + \phi[u_i u_{i+1}]) W(u_{i+2} \otimes \ldots \otimes u_{k-1}).
\end{split}
\]
For $k = i+2$,
\[
\begin{split}
X(u_i) W(u_{i+1})
& = W(u_i \otimes u_{i+1}) + \Lambda(u_i) W(u_{i+1}) + \phi[u_i u_{i+1}].
\end{split}
\]
For $k = i+1$,
\[
\begin{split}
X(u_i) 1
& = W(u_i)
\end{split}
\]

By comparing matrix entries, we can see that this implies for $n \geq 2$,
\begin{equation*}
X W_n = W_{n+1} + A^0 W_n + (\Gamma + \Phi) W_{n-1},
\end{equation*}
and
\begin{equation}
\label{Eq:Matricial-recursion}
X W = (W - I) + A^0 (W - I) + \Phi + (\Gamma + \Phi) (W - I). \qedhere
\end{equation}
\end{proof}

Since $B-X$ is an upper-triangular matrix with only 1s along the main diagonal, its inverse exists for a finite family of $\{u_i\}$, and in a strictly formal sense for an infinite family. In the next section, we will discuss conditions under which $(B-X)^{-1}$, and $W$, are bounded operators on $\mc{F}_{\gamma, \phi}(\mc{B})\otimes \ell^2$.

Since $W$ contains all the information about the multivariate $W$'s, it should be considered as a generating function $W(u_i : i \in \mf{N})$. By constructing certain corresponding operators on a matricial Fock space construction, we can view $W$ as a more traditional generating function of these operators.

\begin{Defn}\label{Defn:BtensorDspace}
Let $\mc{D}$ be a $\ast$-algebra; below we will take $\mc{D} = \mc{L}(\ell^2)$. $\mc{B} \otimes \mc{D}$ is naturally a $\mc{D}$-bimodule. On the Fock space
\[
\bigoplus_{n=0}^\infty \mc{B}^{\otimes n} \otimes \mc{D} \simeq (\mc{B} \otimes \mc{D})^{\otimes_{\mc{D}} n},
\]
we consider operators indexed by elements of $\mc{B} \otimes \mc{D}$
\[
a^+(b \otimes d) (u_1 \otimes \ldots \otimes u_n \otimes d') = b \otimes u_1 \otimes \ldots \otimes u_n \otimes d d',
\]
\[
a^-(b \otimes d) (u_1 \otimes \ldots \otimes u_n \otimes d') = \gamma[b u_1] u_2 \otimes \ldots \otimes u_n \otimes d d',
\]
\[
a^-(b \otimes d) (u_1 \otimes d') = \phi[b u_1] d d',
\]
\[
a^0(b \otimes d) (u_1 \otimes \ldots \otimes u_n \otimes d') = \Lambda(b) u_1 \otimes  u_2 \otimes \ldots \otimes u_n \otimes d d',
\]
\[
a^-(b \otimes d) (d') = a^0(b \otimes d)(d') = 0.
\]
Then any operator in the algebra generated by $\set{a^+(b \otimes d), a^-(b \otimes d), a^0(b \otimes d) : b \in \mc{B}, d \in \mc{D}}$ is of the form $A \otimes d$, where $A$ acts purely on $\bigoplus_{n=0}^\infty \mc{B}^{\otimes n}$. Moreover, the map
\begin{equation}
\label{Eq:Psi}
\Psi: (A \otimes d) \mapsto \ip{(A \otimes d) (1_{\mc{B}} \otimes 1_{\mc{D}})}{1_{\mc{B}} \otimes 1_{\mc{D}}} = \ip{A 1_{\mc{B}}}{1_{\mc{B}}} d
\end{equation}
is a $\mc{D}$-valued conditional expectation on the algebra generated by $\set{X(b \otimes d) : b \in \mc{B}, d \in \mc{D}}$.
\end{Defn}

\begin{Remark}
\label{Remark:U}
Let $\mc{D} = \mc{L}(\ell^2)$ be the algebra of linear operators on $\ell^2$. Let $\set{u_i : i \in \mf{N}} \subset \mc{B}$, and define the matrix $U \in \mc{B}\otimes \mc{L}(\ell^2)$ by $U = \sum_{i\geq 1} u_i \otimes E_{i, i+1}$. Then the objects from Proposition~\ref{Prop:Infinite-W} are in fact
\[
\Phi = (\phi\otimes I)[U^2], \quad \Gamma = (\gamma \otimes I)[U^2], \quad A^0 = (\Lambda \otimes I)[U],
\]
and $X = a^+(U) + a^-(U) + a^0(U)$. If, in addition, \begin{equation}
\label{Eq:Bounded}
\sup_i \norm{u_i} < \infty,
\end{equation}
then $U \in \mc{B} \otimes \mf{B}(\ell^2)$, where $\mf{B}(\ell^2)$ denotes the algebra of bounded operators.
\end{Remark}

Before interpreting $W(U)$ as a generating function, we will first establish the analogous results for moments and cumulants. The formulas below are similar to, but once again different from, those in Section~6.3 in \cite{Popa-Vinnikov-NC-functions}. The first of these is immediate.

\begin{Lemma}
For $\Psi$ defined in \eqref{Eq:Psi} for the case of $\mc{D} = \mc{L}(\ell^2)$,
\begin{equation}
\Psi[X^n] = \sum_{i=1}^\infty \ip{X(u_i) ... X(u_{i+n-1}) \Omega}{\Omega} \ E_{i, i + n}.
\end{equation}
\end{Lemma}

Given a non-commutative operator-valued probability space
\[
(\mathrm{Alg}(X(b \otimes d) : b \in \mc{B}, d \in \mc{D} = \mc{L}(\ell^2)), \mc{D}, \Psi),
\]
we may define $\mc{D}$-valued free cumulants $\mc{R}[d_0 X, d_1 X, \ldots, d_{n-1} X d_n]$ as in Chapter~4 of \cite{SpeHab}. However, we will only be interested in these for $d_0 = d_1 = \ldots = d_n = 1$. In this case, we have the relation

\begin{Cor}
\label{Cor:MatricialCumulants}
\begin{equation}
\mc{R}_n[X, ..., X] = \sum_{i = 1}^\infty R[X(u_i), ..., X(u_{i + n - 1})] E_{i, i + n}.
\end{equation}
\end{Cor}

\begin{proof}
By induction. For $n=1$,
\begin{equation*}
\mc{R}_1[X]_{i,j} = \Phi[X]_{i,j} = \phi[X(u_i)], \text{ for }j-i=1.
\end{equation*}
For $n=2$,
\begin{equation*}
\Phi[X^2] = \mc{R}_2[X,X] + \mc{R}_1[X]^2,
\end{equation*}
so
\begin{equation*}
\mc{R}_2[X,X] = \Phi[X^2] - \mc{R}_1[X]^2 = \Phi[X^2] - \Phi[X]^2,
\end{equation*}
then
\begin{eqnarray*}
\mc{R}_2[X,X]_{i,j}&=& \phi[X(u_i)X(u_{i+1})] - \phi[X(u_i)]\phi[X(u_{i+1})] \\
&=& R_2[X(u_i),X(u_{i+1})], \text{ for }j-i=2.
\end{eqnarray*}
Next, assume this is true for up to $n-1$. Then
\begin{equation*}
\mc{R}_n[X,...,X] = \phi[X^n] - \sum_{\pi\in\mathrm{NC}(n)\backslash \widehat{1}_n} \mc{R}_\pi [X,...,X],
\end{equation*}
where $\mc{R}_\pi[X,...,X] = \prod_{V\in\pi} \mc{R}_{|V|}[X,...,X]$. Applying the inductive hypothesis, the matrices in the right-hand side are nonzero only for the $i,j$th entries such that $j-i=n$, and moreover, after summing the right-hand side, the $i,j$th entry (for such $i,j$) is the cumulant $R[X(u_i),..., X(u_{i+j-1})]$.
\end{proof}

Finally, the next proposition follows by comparing matricial recursions.

\begin{Prop}
\label{Prop:Diagonals-W}
Let $U$ be as in Remark~\ref{Remark:U}. Define the operators $W_n(U)$ on the Fock space $\mc{F}_{\gamma, \phi}(\mc{B}) \otimes \mc{L}(\ell^2)$ by the recursion
\[
\begin{split}
X(U) W_n(U)
& = W_{n+1}(U) + \delta_{n \geq 1} (\Lambda \otimes I)(U) W_n(U) \\
& \quad + \Bigl( \delta_{n \geq 2} (\gamma \otimes I)[U^2] + (\phi \otimes I) [U^2] \Bigr) W_{n-1}(U)
\end{split}
\]
with the initial condition $W_0(U) = I$, and their generating function $W(U) = \sum_{n=0}^\infty W_n(U)$. Then in the setting of Proposition~\ref{Prop:Infinite-W}, $W_n = W_n(U)$ and $W = W(U)$ entrywise. Note that if we only have finitely many non-zero entries $u_i$, boundedness condition~\eqref{Eq:Bounded} holds automatically.
\end{Prop}

\section{Norm estimates and convergence of generating functions}
\label{Section:Norms}

The next two results can be used to estimate the norm of $X(f)$. The following lemma is closely related to Lemma~4 in \cite{BozSpeBM1} and Lemma~1 in \cite{AnsQLevy}, although it is not stated in quite this form in either of those sources.

\begin{Lemma}
\label{Lemma:Undeformed-norm}
Let $\mc{H}$ be a Hilbert space, and $K$ a positive operator on it. Denote $\ip{\vec{\zeta}}{\vec{\xi}}_K = \ip{\vec{\zeta}}{K \vec{\xi}}$ the corresponding deformed inner product. Then for an operator $X$ on $\mc{H}$, denoting by $X^\ast$ its adjoint with respect to the deformed inner product, $\norm{X}_K \leq \sqrt{\norm{X} \norm{X^\ast}}$.
\end{Lemma}

For the remainder of this section, we assume that $\mc{B}$ is a $C^\ast$-algebra, the maps $\phi : \mc{B} \rightarrow \mf{C}$ and $\gamma : \mc{B} \rightarrow \mc{B}$ are bounded, and $\Lambda$ has the special form $\Lambda(u \otimes v) = \Lambda(u) v$ with $\Lambda: \mc{B} \rightarrow \mc{B}$ bounded.

\begin{Prop}\label{Prop:Estimates}
For $b \in \mc{B}$,
\[
\norm{a^+(b)}_{\gamma, \phi} = \norm{a^-(b)}_{\gamma, \phi} \leq \sqrt{\max(\phi[b^\ast b], \norm{(\gamma + \phi)[ b^\ast b]})} \leq \sqrt{\max(\norm{\phi}, \norm{\gamma + \phi})} \ \norm{b}\]
and
\[
\norm{a^0(b)}_{\gamma, \phi} \leq \norm{\Lambda(b)} \leq \norm{\Lambda} \ \norm{b}.
\]
It follows that in this case, the operators in Construction~\ref{Construction:gamma} are well-defined and bounded on $\mc{F}_{\gamma, \phi}(\mc{B})$.
\end{Prop}

\begin{proof}
Since tensors of different length in the Fock space are orthogonal, it suffices to estimate $\norm{a^+(b)}_{\gamma, \phi}$ separately on tensors of fixed length.	For $n=0$,
	\begin{equation*}
	\frac{\langle a^+(b)\Omega,a^+(b)\Omega\rangle_{\gamma,\phi}}{\langle \Omega ,\Omega\rangle_{\gamma,\phi}}
= \langle b ,b\rangle_{\gamma,\phi}
= \phi[b^\ast b] \leq \|\phi\|\|b\|_{\mathcal{B}}^2.
	\end{equation*}
For $n \geq 1$, denote
\[
m(b) (u_1 \otimes \ldots \otimes u_n) = (b u_1) \otimes \ldots \otimes u_n.
\]
and note that $a^-(b^\ast) a^+(b) = m((\gamma + \phi)[b^\ast b])$.

For a fixed tensor $\vec{\xi} \perp \Omega$, denote $\psi_{\vec{\xi}}[u] = \ip{m(u) \vec{\xi}}{\vec{\xi}}_{\gamma, \phi}$. Then $\psi_{\vec{\xi}}$ is linear, and $\psi_{\vec{\xi}}[u^\ast u] = \ip{m(u) \vec{\xi}}{m(u) \vec{\xi}}_{\gamma, \phi} \geq 0$. Therefore
\[
\ip{a^+(b) \vec{\xi}}{a^+(b) \vec{\xi}}_{\gamma, \phi}
= \psi_{\vec{\xi}}[(\gamma + \phi)[b^\ast b]]
\leq \norm{(\gamma + \phi)[b^\ast b]} \psi_{\vec{\xi}}[1]
= \norm{(\gamma + \phi)[b^\ast b]} \norm{\vec{\xi}}_{\gamma, \phi}^2.
\]

	Hence, $\|a^+(b)\|_{\gamma, \phi} \leq \sqrt{\max(\phi[b^\ast b], \norm{(\gamma + \phi)[ b^\ast b]})}$.
	
	Since $a^-(b)$ is the adjoint of $a^+(b)$, their operator norms are equal, so the above inequality also applies to $a^-$.
	
The argument for $a^0(b)$, with $\Lambda$ of the special form, is similar.
\end{proof}

\begin{Prop}\label{Prop:WNorms}
	Denote $c_{\gamma, \phi} = \sqrt{\max(\norm{\phi}, \norm{\gamma + \phi})}$ and
\[
K = \max \left(1, (\sqrt{2} - 1) \left(2 + \frac{\norm{\Lambda}}{c_{\gamma, \phi}} \right) \right).
\]
Then for $(u_n)_{n=1}^\infty \subset \mc{B}$,
	\[
	\norm{W(u_1 \otimes \ldots \otimes u_n)} \leq (\sqrt{2} + 1)^n c_{\gamma, \phi}^n K \norm{u_1} \ldots \norm{u_n}.
	\]
\end{Prop}

\begin{proof}
$\norm{W(\emptyset)} = 1 \leq K$, and by Proposition~\ref{Prop:Estimates},
\[
\norm{W(u_1)} = \norm{X(u_1)} \leq (2 c_{\gamma, \phi} + \norm{\Lambda}) \norm{u_1} \leq (\sqrt{2} + 1) c_{\gamma, \phi} K \norm{u_1}.
\]
For $n=2$,
	\[
	W(u_1 \otimes u_2) = X(u_1) W(u_2) - \Lambda(u_1) W(u_2) - \phi[u_1 u_2] = (a^+(u_1) + a^-(u_1)) W(u_2) - \phi[u_1 u_2],
	\]
and
	\[
\begin{split}
	\norm{W(u_1 \otimes u_2)}
& \leq 2 c_{\gamma, \phi} \ \norm{u_1} \norm{W(u_2)} + \norm{\phi} \ \norm{u_1} \ \norm{u_2},
\end{split}
	\]
while for $n \geq 3$,
	\[
	W(u_1 \otimes u_2 \otimes \ldots \otimes u_n) = (a^+(u_1) + a^-(u_1)) W(u_2 \otimes \ldots \otimes u_n) - (\gamma + \phi)[u_1 u_2] W(u_3 \otimes \ldots \otimes u_n).
	\]
and
	\[
	\norm{W(u_1, u_2, \ldots, u_n)}
	\leq 2 c_{\gamma, \phi} \ \norm{u_1} \norm{W(u_2, \ldots, u_n)} + \norm{\gamma + \phi} \ \norm{u_1} \ \norm{u_2} \ \norm{W(u_3, \ldots, u_n)}.
	\]
Since $(\sqrt{2} + 1)^n = 2 (\sqrt{2} + 1)^{n-1} + (\sqrt{2} + 1)^{n-2}$, the result follows by induction.
\end{proof}

Let's turn our attention back to the infinite matrix $W$ whose structure we established in the previous section, particularly the decomposition into a sum of matrices with a single non-zero diagonal. The following more general proposition will be combined with this fact to obtain conditions under which $W$ is a bounded operator. The following result must be standard, although we do not have a reference for it.

\begin{Lemma}
\label{Lemma:Norm-over-diagonals}
Let $T \in \mf{B} (\mc{H} \otimes \ell^2)$, which we can write as an infinite matrix with entries $T_{ij} \in \mf{B}(\mc{H})$. Then
\[
\norm{T} \leq \sum_{i \in \mf{Z}} \sup_{k \geq \max(1, 1 - i)} \norm{T_{k, k+i}}.
\]
\end{Lemma}

\begin{Cor}
For $\sup_i \norm{u_i} <  \frac{\sqrt{2} - 1}{\sqrt{\max(\norm{\phi}, \norm{\gamma + \phi})}}$,
\begin{enumerate}
\item
$W$ is a bounded operator.
\item
$B - X$ is invertible, and
\[
W = (B - X)^{-1} (B - \Phi).
\]
\end{enumerate}
\end{Cor}
\begin{proof}
For (a), by Proposition~\ref{Prop:WNorms}, $\norm{W_{i, i+n}} \leq K \left((\sqrt{2} + 1) c_{\gamma, \phi} \sup_i \norm{u_i} \right)^n$.
The result follows from Lemma~\ref{Lemma:Norm-over-diagonals}.

For (b), let $Y = B-X-1 = A^0 + \Gamma + \Phi - X$. Note that $Y$ has non-zero entries on only two diagonals:
\begin{eqnarray*}
	Y_{i,i+1} &=& a^+(u_i) + a^-(u_i) \\
	Y_{i,i+2} &=& (\gamma + \phi)[u_i u_{i+1}].
\end{eqnarray*}
So apply Lemma~\ref{Lemma:Norm-over-diagonals} and Proposition~\ref{Prop:Estimates} to get the estimate
\begin{eqnarray*}
	\norm{Y} &\leq& (\norm{a^+} + \norm{a^-})\left(\sup_i \norm{u_i}\right) + \norm{\gamma + \phi} \left(\sup_i \norm{u_i}\right)^2 \\
	&\leq& 2 c_{\gamma, \phi} \left(\sup_i \norm{u_i}\right) + c_{\gamma, \phi}^2 \left(\sup_i \norm{u_i}\right)^2 < 1.
\end{eqnarray*}
Therefore $(B-X)$ is invertible. The formula follows from Proposition~\ref{Prop:Infinite-W}.
\end{proof}

Finally, we return to the question of convergence for the cumulant generating function.

\begin{Prop}\label{Prop:CGF-Convergence}
	Let $K = \max\{\sqrt{\norm{\gamma}}, \norm{\Lambda}\}$. Then $\norm{R_n'[u]} \leq C_n K^n \norm{u}^n$, where $C_n$ is the Catalan number. If $\norm{u}\leq \frac{1}{4K}$, then the generating function $R'(u)$ converges, and thus the cumulant generating function does as well.
\end{Prop}
\begin{proof}
	By Lemma~\ref{Lemma:R'-recursion}, $R'_0 [u] = 1$, $\norm{R'_1[u]} \leq \norm{\Lambda}\norm{u} \leq K \norm{u}$, and
\[
\norm{R'_2[u]} \leq \norm{\Lambda}^2 + \norm{\gamma} \leq 2 K^2 \norm{u}^2.
\]
Recursively,
	\begin{equation*}
\label{Eq:RPrimeNorm}
\begin{split}
	\norm{R'_n[u]} & \leq \sum_{i=0}^{n-2} \norm{R'_i[u]}\norm{R'_{n-i-2}[u]}\norm{\gamma}\norm{u}^2 + \norm{R'_{n-1}[u]}\norm{\Lambda}\norm{u} \\
& \leq \sum_{i=0}^{n-2} C_i C_{n-i-2} K^{n-2} \norm{\gamma} \norm{u}^n + C_{n-1} K^{n-1} \norm{\Lambda} \norm{u}^n \\
& \leq C_n K^n \norm{u}^n
\end{split}
	\end{equation*}
since $C_{n-1} = \sum_{i=0}^{n-2} C_i C_{n-i-2}$ and $C_n \geq 2 C_{n-1}$. So the generating function has norm bounded by
	\begin{equation*}
		\sum_{n=0}^\infty \norm{R'_n[u]} \leq \sum_{n=0}^\infty C_n K^n \norm{u}^n.
	\end{equation*}
A well-known approximation of the Catalan numbers via Stirling's formula is $C_n \sim \frac{4^n}{n^{3/2}\sqrt{\pi}}$. So this series converges by assumption.
\end{proof}

For the generating function $\mc{R}[U] = \sum_{n=0}^{\infty} R_n[X, ..., X]$ for matricial cumulants, we also have the following corollary:

\begin{Cor}
\label{Cor:Norm-cumulant-GF}
Under the conditions of Proposition~\ref{Prop:CGF-Convergence}, with the assumption $\norm{u}\leq \frac{1}{4K}$ replaced with $\sup_{i}\norm{u_i} \leq \frac{1}{4K}$, the infinite matrix $\mc{R}[U]$ is bounded, with
\begin{equation}
\norm{\mc{R}[U]} \leq \sum_{n=2}^{\infty} \norm{\phi} C_{n-2} K^{n-2} \left(\sup_i \norm{u_i}\right)^n .
\end{equation}
\end{Cor}
\begin{proof}
Using Corollary~\ref{Cor:MatricialCumulants}, the $i$th diagonal of $\mc{R}[U]$ has the bound $\norm{\phi} C_{n-2} K^{n-2} \left(\sup_i \norm{u_i}\right)^n$, so the convergence of the sum follows from Lemma~\ref{Lemma:Norm-over-diagonals} and the same argument from the proof of Proposition~\ref{Prop:CGF-Convergence}.
\end{proof}

\section{Traciality}
\label{Section:Trace}

Recall that a state $\psi$ on a non-commutative algebra $\mc{A}$ is \textit{tracial}, or a \textit{trace}, if for all $x,y\in\mc{A}$,
\begin{equation*}
\psi[xy] = \psi[yx].
\end{equation*}

In this section, we give conditions under which the vacuum state is tracial. We start with an auxiliary result.

\begin{Defn}
On $\mc{F}_{\text{alg}}(\mc{B})$, define an anti-linear involution by the linear extension of
\[
S(u_1 \otimes \ldots \otimes u_n) = u_n^\ast \otimes \ldots \otimes u_1^\ast.
\]
For $b \in \mc{B}$, denote $X_r(b) = S X(b^\ast) S$. Explicitly, for $n \geq 2$,
\[
\begin{split}
X_r(u_1 \otimes \ldots \otimes u_n)
& = u_1 \otimes \ldots \otimes u_n \otimes b + u_1 \otimes \ldots \otimes u_{n-1} \otimes \Lambda(b^\ast \otimes u_n^\ast)^\ast \\
&\quad + u_1 \otimes \ldots \otimes u_{n-1} (\gamma + \phi)[u_n b],
\end{split}
\]
with appropriate modifications for $n = 1$ and $n=0$. Denote $\Gamma^{alg}_{\gamma, \Lambda}(\mc{B}, \phi; r) = S \Gamma^{alg}_{\gamma, \Lambda}(\mc{B}, \phi) S$ the algebra generated by $\set{X_r(b) : b \in \mc{B}}$. For $\vec{\xi} \in \mc{F}_{\text{alg}}(\mc{B})$ a simple tensor, denote $W_r(\vec{\xi}) = S W(\vec{\xi}) S$. Then $W_r(S(\vec{\xi})) \Omega = \vec{\xi}$, so that $\Omega$ is cyclic for $\Gamma^{alg}_{\gamma, \Lambda}(\mc{B}, \phi; r)$.
\end{Defn}

\begin{Prop}
\label{Prop:Interacting-traciality}
\

\begin{enumerate}
\item
Suppose that $\Gamma^{alg}_{\gamma, \Lambda}(\mc{B}, \phi)$ and $\Gamma^{alg}_{\gamma, \Lambda}(\mc{B}, \phi; r)$ commute. Then $W$ extends to a linear map on $\mc{F}_{\text{alg}}(\mc{B})$,
\[
\Gamma^{alg}_{\gamma, \Lambda}(\mc{B}, \phi) = \set{W(\vec{\xi}) : \vec{\xi} \in \mc{F}_{\text{alg}}(\mc{B})}, \quad \Gamma^{alg}_{\gamma, \Lambda}(\mc{B}, \phi; r) = \set{W_r(\vec{\xi}) : \vec{\xi} \in \mc{F}_{\text{alg}}(\mc{B})},
\]
and for $\vec{\xi} \in \mc{F}_{\text{alg}}(\mc{B})$,
\[
W(\vec{\xi})^\ast = W(S(\vec{\xi})),
\]
Therefore for $A \in \Gamma^{alg}_{\gamma, \Lambda}(\mc{B}, \phi)$, $S(A \Omega) = A^\ast \Omega$.
\item
In addition to the assumption in (a), suppose that $\phi$ is tracial on $\mc{B}$. Then the vacuum state is tracial on $\Gamma^{alg}_{\gamma, \Lambda}(\mc{B}, \phi)$.
\end{enumerate}
\end{Prop}

\begin{proof}
\

\begin{enumerate}
\item
Since $\Omega$ is cyclic for $\Gamma^{alg}_{\gamma, \Lambda}(\mc{B}, \phi)$ and $\Gamma^{alg}_{\gamma, \Lambda}(\mc{B}, \phi; r)$, it is separating for them (and the von Neumann algebras they generate). Using the recursion and induction, for $\vec{\xi} \in \mc{F}_{\text{alg}}(\mc{B})$
\[
\begin{split}
W(a^+(b)(\vec{\xi})) \Omega
& = X(b) W(\vec{\xi}) \Omega - W(a^0(b) (\vec{\xi})) \Omega - W(a^-(b)(\vec{\xi})) \Omega \\
& = X(b) \vec{\xi} - a^0(b) (\vec{\xi}) - a^-(b)(\vec{\xi}) \\
& = a^+(b)(\vec{\xi}).
\end{split}
\]
In particular, if $\vec{\xi} = 0$, then $W(\vec{\xi}) \Omega = 0$, so  $W(\vec{\xi}) = 0$, and the map $W$ is well defined.
\\
Clearly, $W(\vec{\xi}) \in \Gamma^{alg}_{\gamma, \Lambda}(\mc{B}, \phi)$. On the other hand, using induction on the number of factors, a monomial $X(u_0) X(u_1) \ldots X(u_n)$ is a linear combination of terms of the form $X(u_0) W(\vec{\xi})$, each of which is a linear combination of the terms of the form $W(\vec{\xi'})$ by the recursion.

Next, note that
\[
X_r(b) \Omega = X(b) \Omega
\]
and
\[
W(b)^\ast = X(b)^\ast = X(b^\ast) = W(b^\ast).
\]
Using induction,
\[
\begin{split}
W(a^+(b)(\vec{\xi}))^\ast
& = W(\vec{\xi})^\ast X(b)^\ast - W(a^0(b) (\vec{\xi}))^\ast - W(a^-(b)(\vec{\xi}))^\ast \\
& = W(S(\vec{\xi})) X(b^\ast) - W(S(a^0(b) (\vec{\xi}))) - W(S(a^-(b)(\vec{\xi}))).
\end{split}
\]
On the other hand, if $X_r(b)$ commutes with $\Gamma^{alg}_{\gamma, \Lambda}(\mc{B}, \phi)$,
\[
\begin{split}
W(S(\vec{\xi})) X(b^\ast) \Omega
& = W(S(\vec{\xi})) X_r(b^\ast) \Omega
= X_r(b^\ast) W(S(\vec{\xi})) \Omega \\
& = S X(b) S S(\vec{\xi})
= S X(b) (\vec{\xi}).
\end{split}
\]
Since $\Omega$ is separating for $\Gamma^{alg}_{\gamma, \Lambda}(\mc{B}, \phi)$, this implies that
\[
W(S(\vec{\xi})) X(b^\ast)
= W(S X(b) (\vec{\xi}))
= W(S(a^+(b)(\vec{\xi}))) + W(S(a^0(b)(\vec{\xi}))) + W(S(a^-(b)(\vec{\xi}))).
\]
Thus
\[
W(a^+(b)(\vec{\xi}))^\ast = W(S(a^+(b)(\vec{\xi}))).
\]
\item
We first show that $X_r(b)^\ast \Omega = X(b)^\ast \Omega$. Indeed, for $\vec{\xi} \in \mc{B}^{\otimes n}$, $\ip{X_r(b) (\vec{\xi})}{\Omega} = 0$ if $n \neq 1$. For $u \in \mc{B}$, using the fact that $\phi$ is a trace,
\[
\begin{split}
\ip{X_r(b) (u)}{\Omega}
& = \ip{S a^-(b^\ast)(S(u))}{\Omega}
= \overline{\phi[b^\ast u]}
= \phi[u b]
= \phi[b u]
= \ip{u}{b^\ast}
= \ip{u}{X(b)^\ast \Omega}
\end{split}
\]

\br
To check that the vacuum state is tracial on $\Gamma^{alg}_{\gamma, \Lambda}(\mc{B}, \phi)$, it suffices to verify that
\[
\begin{split}
\ip{W(\vec{\xi}) X(b) \Omega}{\Omega}
& = \ip{W(\vec{\xi}) X_r(b) \Omega}{\Omega}
= \ip{X_r(b) W(\vec{\xi}) \Omega}{\Omega} \\
& = \ip{W(\vec{\xi}) \Omega}{X_r(b)^\ast \Omega}
= \ip{W(\vec{\xi}) \Omega}{X(b)^\ast \Omega} \\
& = \ip{X(b) W(\vec{\xi}) \Omega}{\Omega}. \qedhere
\end{split}
\]
\end{enumerate}
\end{proof}

\begin{Thm}\label{Thm:Traciality}
Suppose the operators $X(u) : u \in \mc{B}$ are bounded. Denote
\[
\Gamma_{\gamma, \Lambda}(\mc{B}, \phi) = W^\ast(X(u) : u \in \mc{B}) = W^\ast(X(u) : u \in \mc{B}^{sa}).
\]
Consider the conditions
\begin{equation}
\label{Eq:Trace-star}
\Lambda(v^\ast \otimes u^\ast)^\ast = \Lambda(u \otimes v),
\end{equation}
\begin{equation}
\label{Eq:Trace-Associative}
u \gamma[y v] - \gamma[u y] v = \Lambda(u \otimes \Lambda(y \otimes v)) - \Lambda(\Lambda(u \otimes y) \otimes v),
\end{equation}
\begin{equation}
\label{Eq:Trace-Extra}
\Lambda(u \otimes y \gamma[z v]) - \Lambda(u \otimes y) \gamma[z v] = \Lambda(\gamma[u y] z \otimes v) - \gamma[u y] \Lambda(z \otimes v)
\end{equation}
and
\begin{equation}
\label{Eq:Trace:gamma-self-adjoint}
\phi[\gamma[u] v] = \phi[u \gamma[v]].
\end{equation}
\begin{enumerate}
\item
Each $X_r(u)$ commutes with $\Gamma_{\gamma, \Lambda}(\mc{B}, \phi)$ if and only if conditions \eqref{Eq:Trace-star}, \eqref{Eq:Trace-Associative}, \eqref{Eq:Trace-Extra}, and \eqref{Eq:Trace:gamma-self-adjoint} hold. Note that if $\gamma$ is scalar-valued, the third condition is true automatically.
\item
If the vacuum state is tracial on $\Gamma_{\gamma, \Lambda}(\mc{B}, \phi)$, then conditions \eqref{Eq:Trace-star}, \eqref{Eq:Trace-Associative}, and \eqref{Eq:Trace-Extra} hold, and $\phi$ is tracial.
\item
If $\phi$ is tracial and each $X_r(u)$ commutes with $\Gamma_{\gamma, \Lambda}(\mc{B}, \phi)$, then the vacuum state is tracial on $\Gamma_{\gamma, \Lambda}(\mc{B}, \phi)$.
\end{enumerate}
Under the assumptions in (c), all the conclusions in Proposition~\ref{Prop:Interacting-traciality} hold, the map $S$ extends to an anti-linear isometry on $\mc{F}_{\gamma, \phi}(\mc{B})$, and the map $A \mapsto S A^\ast S$ implements the canonical anti-isomorphism between $\Gamma_{\gamma, \Lambda}(\mc{B}, \phi)$ and its commutant.
\end{Thm}

\begin{proof}
For part (a), note that since we are working on a Fock space with depth two action, to show that $X(u) X_r(v) = X_r(v) X(u)$ it suffices to consider their actions on tensors of length $0$, $1$, and $2$. A calculation shows that this is equivalent to conditions \eqref{Eq:Trace-star}-\eqref{Eq:Trace:gamma-self-adjoint} together with the previously assumed \eqref{Eq:a0-symmetric}.

For part (b), if the vacuum state is tracial, the joint free cumulants are cyclically symmetric. Using Proposition~\ref{Prop:FreeCumulants} for cumulants of order up to $5$, we obtain the conditions \eqref{Eq:Trace-star}-\eqref{Eq:Trace-Extra} and traciality of $\phi$.

Part (c), and the rest of the statements, follow from Proposition~\ref{Prop:Interacting-traciality} and general theory.
\end{proof}

\section{Examples}
\label{Sec:Examples}

In this section we show how the motivating examples, and their generalizations, fit into the general setting of this paper.

\begin{Example}
\label{Example:Bozejko}
Let
\[
\mc{B} = L^\infty([0,1]), \quad \phi[f] = \int_0^1 f(x) \,dx,
\]
and $\eta$ positive a.e. and $\lambda$ a real-valued function. Define
\[
\gamma[f](x) = \eta(x) f(x), \quad \Lambda(b \otimes f)(x) = \lambda(x) b(x) f(x).
\]
Then the construction is closely related to the setting in \cite{Bozejko-Lytvynov-Meixner-I,Bozejko-Lytvynov-Meixner-II} (see also \cite{Bozejko-Lytvynov-Rodionova-Anyon}).

More generally, for an algebra $\mc{B}$, let $\eta, \lambda$ be central elements in $\mc{B}$ such that $\eta + 1$ is positive and $\lambda$ self-adjoint. Then we may take
\[
\gamma[f] = \eta f, \quad \Lambda(b \otimes f) = \lambda b f.
\]
If $\eta + 1$ is invertible, the inner product is positive definite. The operators are typically unbounded. The vacuum state is always tracial.

The free cumulants are given by an explicit formula
\[
R[X(f_1), \ldots, X(f_n)] = \sum_{\pi \in \cNC} \phi[\eta^{\abs{\pi} - 1} \lambda^{n - 2 \abs{\pi}} f_1 f_2 \ldots f_n].
\]
In particular, in this case $R'[(X(f_2), \ldots, X(f_{n-1})]$ may be identified with an element of $\mc{B}$, and their generating function satisfies a quadratic relation
\begin{equation} \label{Eq:BozekjoRPrime}
R'(f)
= 1 + \lambda R'(f) f + \eta R'(f) f R'(f).
\end{equation}
In the case where $\mc{B} = L^\infty([0,1])$, we may take $R'(f)$ to be the operator of pointwise multiplication by a function $R'(f)(x)$ which satisfies
\[
R'(f)(x)
= 1 + R'(f)(x) \lambda(x) f(x) + R'(f)(x)^2 \eta(x) f(x),
\]
and so is a solution of a quadratic equation for each $x$. The free cumulant generating function can then be found via
\[
R(f) = \int f(x)^2 R'(f)(x) \,dx.
\]
Finally, if $f g = g f = f^\ast g = f g^\ast = 0$, then $X(f)$ and $X(g)$ are free, and if $f$ is self-adjoint with $\eta f = \lambda f = 0$, then $X(f)$ is semicircular. See also the end of the section for related results.
\end{Example}

\begin{Example}
\label{Example:Lenczewski}
Let
\[
\mc{B} = L^\infty([0,1]), \quad \phi[f] = \int_0^1 f(x) \,dx,
\]
and
\[
w(x, y) =
\begin{cases}
p & \text{if } 0 < x < y, \\
q & \text{if } 0 < y < x, \\
1 & \text{otherwise}
\end{cases}
\]
for $p,q \geq 0$. Define
\[
\gamma[f](x) = \int w(x,y) f(y) \,dy
\]
and
\[
\Lambda(b \otimes f)(x) = f(x) \int w(x,y) b(y) \,dy + f(y) b(y).
\]
We obtain the construction from \cite{Lenczewski-Salapata-Kesten}.

The specific form of a function $w$ above allowed for explicit combinatorial calculations, but $\gamma$ can be defined by the same formula as long as $(w + 1)$ is non-negative a.e., and it is natural to take
\[
\Lambda(b \otimes f)(x) = \int \lambda(x, y, z) b(y) f(z) \,dy \,dz
\]
or
\[
\Lambda(b)(x) = \int b(y) \lambda(x, dy).
\]
In the first case, necessary conditions for the vacuum state to be tracial imply that $\lambda(x, y, z)$ is conjugate-symmetric in its arguments, and $w(x,y)$ is symmetric in its arguments. These conditions are not sufficient, and typically the vacuum state is not tracial.

In the second case, we may identify $R'(f)$ with a function satisfying
\[
R'(f)(x) = \frac{1}{1 - \int \lambda(x, y) f(y) \,dy - \int w(x,y) R'(f)(y) f(y)^2 \,dy}.
\]

One can generalize this construction further. Instead of the map $\gamma$, it suffices to use a bilinear map $\ip{\cdot}{\cdot}_\gamma: \mc{B} \otimes \mc{B} \rightarrow \mc{B}$ which is star-linear in each argument and such that $\ip{\cdot}{\cdot}_\gamma + \ip{\cdot}{\cdot}_\phi$ is positive semi-definite. Let $w \in \mc{B} \otimes \mc{B}$ satisfy $w + (1 \otimes 1) \geq 0$, and define
\[
\ip{f}{g}_\gamma = (\phi \otimes I)[(f \otimes 1) w (g \otimes 1)].
\]
Then the corresponding inner product on the Fock space is positive semi-definite. If $w + 1 \otimes 1$ is invertible, the inner product is positive definite.

Note that
\[
\ip{\vec{\zeta}}{\vec{\xi}}_{\gamma, \phi} = \ip{\vec{\zeta}}{K \vec{\xi}},
\]
where $K$ is the multiplication operator by
\[
[(1 + w) \otimes 1^{\otimes (n-2)}] \cdot [1 \otimes (1 + w) \otimes 1^{\otimes (n-3)}] \cdot \ldots \cdot [1^{\otimes (n-2)} \otimes (1 + w)].
\]
In the commutative case we may identify
\[
K(s_1, s_2, \ldots s_n) = (1 + w(s_1, s_2)) \ldots (1 + w(s_{n-1}, s_n)).
\]
So Lemma~\ref{Lemma:Undeformed-norm} applies. Moreover, $a^-(b) = \ell^\ast(b) [(1 + w) \otimes 1^{\otimes (n-2)}]$, where $\ell^\ast$ is the free annihilation operator. It follows that for $w \in \mc{B} \otimes_{\min} \mc{B}$ and $\norm{\Lambda(b \otimes b_1)}_\phi \leq \norm{\Lambda} \ \norm{b} \ \norm{b_1}_\phi$,
\[
\norm{a^-(b)}_{\gamma, \phi} = \norm{a^+(b)}_{\gamma, \phi} \leq \norm{b}_\phi \sqrt{\norm{w} + 1}, \quad \norm{a^0(b)}_{\gamma, \phi} \leq \norm{\Lambda} \ \norm{b}.
\]
In particular, the operators extend to $\mc{F}_{\gamma, \phi}(\mc{B})$.
\end{Example}

\begin{Example}
\label{Example:MA}
Let $\mc{B} = \mf{R}^d$ with component-wise multiplication and the standard basis $\set{e_i}_{i=1}^d$,
\[
\phi[e_j] = 1, \quad \gamma[e_j] = \sum_{i=1}^d C_{ij} e_i, \quad \Lambda(e_i \otimes e_j) = \sum_{k=1}^d B_{ij}^k e_k,
\]
with $C_{ij} \geq -1$ and certain relations between the coefficients. This is the setting of \cite{AnsFree-Meixner} (possibly with slightly different notation). Embedding $\mf{R}^d$ into $L^\infty[0,1]$ using functions constant on cubes, relates this example to the preceding one.

If all $C_{ij} > -1$, the inner product is positive definite. The operators are clearly bounded, and so extend to $\mc{F}_{\gamma, \phi}(\mc{B})$. Particular parameters for which the vacuum state is tracial were given in \cite{AnsFree-Meixner}.
\end{Example}

\begin{Example}
\label{Example:Scalar-gamma}
Choose $\mc{B}$, $\phi$, $\Lambda$ to be general, but let $\gamma = \psi - \phi$ be scalar-valued, where $\psi$ is positive semi-definite. Then the inner product simplifies to
\[
\ip{f_1 \otimes \ldots \otimes f_n}{g_1 \otimes \ldots \otimes g_k}_{\gamma,\phi}
= \delta_{n=k} \phi[g_n^\ast f_n] \prod_{i=1}^{n-1} \psi[g_i^\ast f_i],
\]
so that
\[
\mc{F}_{\gamma, \phi}(\mc{B}) \simeq \mf{C} \Omega \oplus \left( L^2(\mc{B}, \phi) \otimes \mc{F}(L^2(\mc{B}, \psi)) \right)
\]
with the standard inner product. A necessary condition for the vacuum state to be tracial is that $\psi$ is a multiple of $\phi$. For $\psi \neq \phi$ and $\Lambda = 0$, the vacuum state is not tracial. In \cite{Ans-Mashburn-Centered}, we investigate a choice of $\Lambda$ which does lead to a tracial vacuum state.
\end{Example}

\begin{Example}
\label{Example:Free-Poisson}
Let $\gamma = 0$ and $\Lambda(f \otimes g) = f g$. For any $(\mc{B}, \phi)$, the corresponding algebra of operators on the Fock space is the free (compound) Poisson algebra, a particular case of the constructions in the Appendix of \cite{AnsQLevy} (for $q = 0$) or Proposition~23 in \cite{Ans-Wick-products} (for the scalar case). See also \cite{GloSchurSpe}. In particular, for $f \in \mc{B}^{sa}$, $X(f)$ has a centered free compound Poisson distribution, and a centered free Poisson distribution if $f$ is a projection. Moreover, if $f \cdot g = g \cdot f = 0$ then $X(f)$ and $X(g)$ are free.
\end{Example}

We finish the paper with two structure results under the assumption that the vacuum state is tracial.

\begin{Prop}
Assume that $\mc{B}$ is unital, and the vacuum state is tracial. Suppose that $\Lambda(u \otimes v) = \Lambda(u) v$. Then $\Gamma^{alg}_{\gamma, \Lambda}(\mc{B}, \phi)$ is as in Example~\ref{Example:Bozejko}, with $\Lambda(u) = \lambda u$ for some central $\lambda$, and $\gamma[u] = \eta u$ for some central $u$.
\end{Prop}

\begin{proof}
If $\Lambda(u \otimes v) = \Lambda(u) v$, then conditions \eqref{Eq:a0-symmetric}, \eqref{Eq:Trace-star}, and \eqref{Eq:Trace-Associative} imply that $u \Lambda(v) = \Lambda(u) v$ and $u \gamma[y v] = \gamma[u y] v$. Therefore $\Lambda(u) = \lambda u = u \lambda$ for $\lambda = \Lambda(1)$, and $\gamma[u] = \eta u = u \eta$ for $\eta = \gamma[1]$.
\end{proof}

\begin{Thm}
Assume that $\mc{B}$ is unital, and the vacuum state is tracial. Suppose that $\gamma[f] = \eta f$ for $\eta$ central, and for some $M$, and all $f, g \in \mc{B}$,
\begin{equation}
\label{Eq:Hilbert-algebra}
\ip{\Lambda(f \otimes g)}{\Lambda(f \otimes g)} \leq M \norm{f}_{\mc{B}}^2 \ip{g}{g}.
\end{equation}
Then we have an orthogonal decomposition
\[
\mc{H} = (\mc{B}, \ip{\cdot}{\cdot}_\phi) = \mc{N} \oplus \mc{N}^\perp = \mc{Z} \oplus \mc{P} \oplus \mc{N}^\perp,
\]
such that
\begin{itemize}
\item
$\Gamma^{alg}_{\gamma, \Lambda}(\mc{N}^\perp, \phi)$ is as in Example~\ref{Example:Bozejko}, with $\Lambda(f \otimes g) = \lambda f g$ for some central $\lambda$.
\item
$\Gamma^{alg}_{\gamma, \Lambda}(\mc{Z}, \phi)$ is generated by free semicircular elements.
\item
$\Gamma^{alg}_{\gamma, \Lambda}(\mc{P}, \phi)$ is a free Poisson algebra as in Example~\ref{Example:Free-Poisson}, where the multiplication on $\mc{P}$ is given by $\Lambda$ (rather than inherited from $\mc{B}$).
\end{itemize}
Moreover, the subalgebras $\Gamma^{alg}_{\gamma, \Lambda}(\mc{N}^\perp, \phi)$, $\Gamma^{alg}_{\gamma, \Lambda}(\mc{Z}, \phi)$, and $\Gamma^{alg}_{\gamma, \Lambda}(\mc{P}, \phi)$ are free.
\end{Thm}

\begin{proof}
Denote by $\mc{H}$ the inner product space $\mc{B}$ with the involution and the inner product $\ip{f}{g} = \phi[g^\ast f]$ satisfying the relation
\[
\ip{f}{g} = \ip{g^\ast}{f^\ast}.
\]
Denote
\begin{equation}
\label{Eq:Lambda-product}
f \cdot g = \Lambda(f \otimes g).
\end{equation}

The notation is meant to suggest that we consider the binary operation $\Lambda$ as a multiplication on $\mc{H}$ (different from a multiplication it may have inherited from $\mc{B}$). Indeed, equations \eqref{Eq:Trace-star}, \eqref{Eq:Trace-Associative}, and the first condition in \eqref{Eq:a0-symmetric} say that
\[
(f \cdot g)^\ast = g^\ast \cdot f^\ast, \quad (f \cdot g) \cdot h = f \cdot (g \cdot h), \quad \ip{f \cdot g}{h} = \ip{g}{f^\ast \cdot h}.
\]
That is, $(\mc{H}, \cdot, \ast)$ is an associative star-algebra (linearity/distributivity is immediate), which is represented on $\mc{H}$ by left multiplication operators.

Second condition in equation~\eqref{Eq:a0-symmetric} reads
\begin{equation}
\label{Eq:Switch}
\eta g^\ast \Lambda(b \otimes f) = \eta \Lambda(b^\ast \otimes g)^\ast f = \eta \Lambda(g^\ast \otimes b) f.
\end{equation}
Therefore $\eta \Lambda(b \otimes f) = \eta \lambda b f$, where $\lambda = \Lambda(1 \otimes 1)$.

Let
\[
\mc{N} = \set{ f \in \mc{B} : \eta f = 0}.
\]
Then $\mc{N}$ is a self-adjoint subspace and an ideal. Restricted to $\mc{N}^\perp$, multiplication by $\eta$ is injective. Since by \eqref{Eq:Trace-star}, $\eta \lambda b g = \eta b g \lambda^\ast$, it follows that $\lambda$, restricted to $\mc{N}^\perp$, is self-adjoint and central. $\mc{N}^\perp$ is self-adjoint and a left ideal with respect to $\cdot$, so $\Gamma^{alg}_{\gamma, \Lambda}(\mc{N}^\perp, \phi)$ is a star-subalgebra.

Let $f_1, \ldots, f_n \in \mc{N} \cup \mc{N}^\perp$, with at least one element from $\mc{N}$. It follows from equation~\eqref{Eq:Switch} that $\mc{N}$ is also an ideal with respect to $\cdot$. Combined with the definition of $\mc{N}$, this implies that in the expansion for the joint free cumulant in Proposition~\ref{Prop:FreeCumulants}, elements of $\mc{N}$ can only appear in the (unique) outer block of $\pi$. Since each inner block is a multiple of $\eta$, it follows that in fact, $\pi$ has a single block so that
\begin{equation}
\label{Eq:Free-Cumulant-Inner}
R[X(f_1), \ldots, X(f_n)] = \ip{f_1}{f_2 \cdot \ldots \cdot f_n}.
\end{equation}
Since the vacuum state is tracial, the free cumulants are cyclically symmetric. So if
\[
f_1, \ldots, f_n \in \mc{N} \cup \mc{N}^\perp
\]
with at least one element from each of $\mc{N}, \mc{N}^\perp$, we may assume that $f_1 \in \mc{N}^\perp$ and $f_2 \cdot \ldots \cdot f_n \in \mc{N}$. Then their inner product is zero, which implies freeness by \eqref{Eq:Free-Cumulant-Inner}.

Next, denote
\[
\mc{Z} = \set{f \in \mc{N} : \forall g \in \mc{N}, g \cdot f = 0}
\]
and
\[
\mc{P} = \overline{\Span(\set{ f \cdot g : f, g \in \mc{N}})}.
\]
Clearly both of them are subspaces (and in fact ideals with respect to $\cdot$). Also, using the properties of $\cdot$ at the beginning of the proof, one sees that $\mc{Z}$ and $\mc{P}$ are orthogonal complements of each other, and $\mc{N} = \mc{Z} \oplus \mc{P}$. $\mc{P}$ is clearly self-adjoint, therefore so is $\mc{Z}$.  %Finally, by the uniqueness of the direct sum decomposition it follows that if $f = z + p$ is self-adjoint, with $z \in \mc{Z}$ and $p \in \mc{P}$, then $z$ and $p$ are also self-adjoint.

For $f_1, \ldots, f_n \in \mc{Z} \cup \mc{P}$ with at least one of them in $\mc{Z}$, using the cyclic property again,
\[
R[X(f_1), \ldots, X(f_n)] = \ip{f_1}{f_2 \cdot \ldots \cdot f_n} = \delta_{n=2} \ip{f_1}{f_2}
\]
and is zero unless both $f_1, f_2 \in \mc{Z}$. It follows that for $f \in \mc{Z}^{sa}$, $X(f)$ are semicircular, with orthogonal $f$ corresponding to free $X(f)$, and are also free from $X(g)$ with $g \in \mc{P}$.

Under the boundedness assumption~\eqref{Eq:Hilbert-algebra}, $\mc{P}$ satisfies all the axioms of a Hilbert algebra (Chapters 5 and 6 of \cite{Dixmier-vN}). Then we can complete $\mc{P}$ to a von Neumann algebra represented on $\mc{N}$ (since it acts as zero on $\mc{Z}$, it is also represented on $\mc{P}$) by left multiplication. Moreover, we have a semi-finite trace $\psi$ on this von Neumann algebra which is finite on $\mc{P}$ and satisfies $\ip{f}{g} = \psi[g^\ast \cdot f]$. So in this case for self-adjoint elements and $n \geq 2$,
\[
R[X(f_1), \ldots, X(f_n)] = \psi[f_1 \cdot f_2 \cdot \ldots \cdot f_n].
\]
It follows that $\Gamma^{alg}_{\gamma, \Lambda}(\mc{P}, \phi)$ is the free compound Poisson algebra of $(\mc{P}, \psi)$.
%Careful: $R[X(f)] = \phi[f]$ not $\psi[f]$.
\end{proof}

\def\cprime{$'$} \def\cprime{$'$}
\providecommand{\bysame}{\leavevmode\hbox to3em{\hrulefill}\thinspace}
\providecommand{\MR}{\relax\ifhmode\unskip\space\fi MR }
\providecommand{\MRhref}[2]{%
  \href{http://www.ams.org/mathscinet-getitem?mr=#1}{#2}
}
\providecommand{\href}[2]{#2}

%\bibliographystyle{amsalpha}
%\bibliography{bibdata}

\begin{thebibliography}{BLR15}

\bibitem[AM22]{Ans-Mashburn-Centered}
Michael Anshelevich and Jacob Mashburn, in preparation, 2022.

\bibitem[Ans04]{AnsQLevy}
Michael Anshelevich, \emph{{$q$}-{L}\'evy processes}, J. Reine Angew. Math.
  \textbf{576} (2004), 181--207. \MR{2099204 (2006c:46052)}

\bibitem[Ans07]{AnsFree-Meixner}
\bysame, \emph{Free {M}eixner states}, Comm. Math. Phys. \textbf{276} (2007),
  no.~3, 863--899. \MR{2350440 (2009b:81106)}

\bibitem[Ans20]{Ans-Wick-products}
\bysame, \emph{Product formulas on posets, {W}ick products, and a correction
  for the {$q$}-{P}oisson process}, Indiana Univ. Math. J. \textbf{69} (2020),
  no.~6, 2129--2169. \MR{4170089}

\bibitem[AW18]{Ans-Williams-Jacobi}
Michael Anshelevich and John~D. Williams, \emph{Operator-valued {J}acobi
  parameters and examples of operator-valued distributions}, Bulletin des
  Sciences Math{\'e}matiques \textbf{145} (2018), 1--37.

\bibitem[BB06]{Boz-Bryc}
Marek Bo{\.z}ejko and W{\l}odzimierz Bryc, \emph{On a class of free {L}\'evy
  laws related to a regression problem}, J. Funct. Anal. \textbf{236} (2006),
  no.~1, 59--77. \MR{2227129 (2007a:46071)}

\bibitem[BL09]{Bozejko-Lytvynov-Meixner-I}
Marek Bo{\.z}ejko and Eugene Lytvynov, \emph{Meixner class of non-commutative
  generalized stochastic processes with freely independent values. {I}. {A}
  characterization}, Comm. Math. Phys. \textbf{292} (2009), no.~1, 99--129.
  \MR{2540072 (2010h:46104)}

\bibitem[BL11]{Bozejko-Lytvynov-Meixner-II}
\bysame, \emph{Meixner class of non-commutative generalized stochastic
  processes with freely independent values {II}. {T}he generating function},
  Comm. Math. Phys. \textbf{302} (2011), no.~2, 425--451. \MR{2770019
  (2011m:46118)}

\bibitem[BLR15]{Bozejko-Lytvynov-Rodionova-Anyon}
M.~Bozhe{\u\i}ko, E.~V. Litvinov, and I.~V. Rodionova, \emph{An extended anyon
  {F}ock space and non-commutative {M}eixner-type orthogonal polynomials in the
  infinite-dimensional case}, Uspekhi Mat. Nauk \textbf{70} (2015), no.~5(425),
  75--120. \MR{3438555}

\bibitem[BN08]{Belinschi-Nica-Eta}
Serban~T. Belinschi and Alexandru Nica, \emph{{$\eta$}-series and a {B}oolean
  {B}ercovici-{P}ata bijection for bounded {$k$}-tuples}, Adv. Math.
  \textbf{217} (2008), no.~1, 1--41. \MR{2357321}

\bibitem[BS91]{BozSpeBM1}
Marek Bo{\.z}ejko and Roland Speicher, \emph{An example of a generalized
  {B}rownian motion}, Comm. Math. Phys. \textbf{137} (1991), no.~3, 519--531.
  \MR{1105428 (92m:46096)}

\bibitem[Dix81]{Dixmier-vN}
Jacques Dixmier, \emph{von {N}eumann algebras}, North-Holland Mathematical
  Library, vol.~27, North-Holland Publishing Co., Amsterdam-New York, 1981,
  With a preface by E. C. Lance, Translated from the second French edition by
  F. Jellett. \MR{641217}

\bibitem[GSS92]{GloSchurSpe}
Peter Glockner, Michael Sch{\"u}rmann, and Roland Speicher, \emph{Realization
  of free white noises}, Arch. Math. (Basel) \textbf{58} (1992), no.~4,
  407--416. \MR{93e:46075}

\bibitem[LS08]{Lenczewski-Salapata-Kesten}
Romuald Lenczewski and Rafa{\l} Sa{\l}apata, \emph{Noncommutative {B}rownian
  motions associated with {K}esten distributions and related {P}oisson
  processes}, Infin. Dimens. Anal. Quantum Probab. Relat. Top. \textbf{11}
  (2008), no.~3, 351--375. \MR{2446514 (2009i:46120)}

\bibitem[PV13]{Popa-Vinnikov-NC-functions}
Mihai Popa and Victor Vinnikov, \emph{Non-commutative functions and the
  non-commutative free {L}\'evy-{H}in\v cin formula}, Adv. Math. \textbf{236}
  (2013), 131--157. \MR{3019719}

\bibitem[Shl99]{Shl-A-semicircular}
Dimitri Shlyakhtenko, \emph{{$A$}-valued semicircular systems}, J. Funct. Anal.
  \textbf{166} (1999), no.~1, 1--47. \MR{1704661 (2000j:46124)}

\bibitem[Spe98]{SpeHab}
Roland Speicher, \emph{Combinatorial theory of the free product with
  amalgamation and operator-valued free probability theory}, Mem. Amer. Math.
  Soc. \textbf{132} (1998), no.~627, x+88. \MR{1407898 (98i:46071)}

\end{thebibliography}

\end{document}